\documentclass[11pt, a4paper]{article}
\usepackage{amssymb,amsmath,amsthm}
\usepackage{fancyhdr}
\usepackage[british]{babel}
\usepackage{enumitem}
\usepackage{color}
\usepackage[title]{appendix}
\usepackage{amsmath,amsthm}

\usepackage{hyperref}
\usepackage[margin=2.5cm]{geometry}

\newtheorem{theorem}{Theorem}[section]
\newtheorem{prop}[theorem]{Proposition}
\newtheorem{lemma}[theorem]{Lemma}

\theoremstyle{definition}

\newtheorem{question}{Question}

\newtheorem{defn}[theorem]{Definition}

\newtheorem{claim}[theorem]{Claim}
\newtheorem{fact}[theorem]{Fact}

\newcommand{\btheorem}{\begin{theorem}}
\newcommand{\etheorem}{\end{theorem}}
\newcommand{\bconjecture}{\begin{conjecture}}
\newcommand{\econjecture}{\end{conjecture}}
\newcommand{\bproposition}{\begin{proposition}}
\newcommand{\eproposition}{\end{proposition}}
\newcommand{\bdefinition}{\begin{definition}}
\newcommand{\edefinition}{\end{definition}}
\newcommand{\bcorollary}{\begin{corollary}}
\newcommand{\ecorollary}{\end{corollary}}
\newcommand{\bproof}{\begin{proof}}
\newcommand{\eproof}{\end{proof}}
\newcommand{\bclaim}{\begin{claim}}
\newcommand{\eclaim}{\end{claim}}
\newcommand{\bquestion}{\begin{question}}
\newcommand{\equestion}{\end{question}}
\newcommand{\bfact}{\begin{fact}}
\newcommand{\efact}{\end{fact}}
\newcommand{\bremark}{\begin{remark}}
\newcommand{\eremark}{\end{remark}}
\newcommand{\eexample}{\end{example}}
\newcommand{\bexample}{\begin{example}}

\newcommand{\elemma}{\end{lemma}}
\newcommand{\blemma}{\begin{lemma}}

\newcommand{\be}{\beta}

\newcommand{\F}{\mathcal{F}}
\newcommand{\C}{\mathcal{C}}

\def \hcf{\mathrm{hcf}}
\newcommand{\eps}{\varepsilon}

\title{Graph tilings in incompatibility systems}

\author{Jie Hu$^a$ \ \ Hao Li$^b$\ \ Yue Wang$^c$\ \ Donglei Yang$^d$\unskip\thanks{Data Science Institute, Shandong University, Shandong, China. Email: {\tt dlyang@sdu.edu.cn}. Supported by the China Postdoctoral Science Foundation (2021T140413), Natural Science Foundation of China (12101365) and Natural Science Foundation of Shandong Province (ZR2021QA029).}
\unskip\\[.5em]
{\small $^a$ Center for Combinatorics and LPMC,}\\
{\small Nankai University, Tianjin 300071, China}\\
{\small $^b$ Laboratoire Interdisciplinaire des Sciences du Num\'erique,}\\
{\small CNRS--Universit\'e Paris-Saclay, Orsay 91405, France}\\
{\small $^c$ School of Mathematics and Statistics,}\\
{\small Shandong Normal University, Jinan 250358, China}\\
{\small $^d$ Data Science Institute,}\\
{\small Shandong University,  Jinan 250100, China}\\
 }
\date{}

\begin{document}
\maketitle


\begin{abstract}
  Given two graphs $H$ and $G$, an \emph{$H$-tiling} of $G$ is a collection of vertex-disjoint copies of $H$ in $G$ and an \emph{$H$-factor} is an $H$-tiling that covers all vertices of $G$. K\"{u}hn and Osthus managed to characterize, up to an additive constant, the minimum degree threshold which forces an $H$-factor in a host graph $G$. In this paper we study a similar tiling problem in a system that is locally bounded. An \emph{incompatibility system} $(G,\mathcal{F})$ consists of a graph $G$ and a family $\mathcal{F}=\{F_v\}_{v\in V(G)}$ over $G$ with $F_v\subseteq \{\{e,e'\}\in {E(G)\choose 2}: e\cap e'=\{v\}\}$. We say that two edges $e,e'\in E(G)$ are \emph{incompatible} if $\{e,e'\}\in F_v$ for some $v\in V(G)$, and otherwise \emph{compatible}. A subgraph $H$ of $G$ is \emph{compatible} if every pair of edges in $H$ are compatible. An incompatibility system $(G,\mathcal{F})$ is \emph{$\Delta$-bounded} if for any vertex $v$ and any edge $e$ incident with $v$, there are at most $\Delta$ members of $F_v$ containing $e$. This notion was partly motivated by a concept of transition system introduced by Kotzig in 1968, and first formulated by Krivelevich, Lee and Sudakov to study the robustness of Hamiltonicity of Dirac graphs.

  We prove that for any $\alpha>0$ and any graph $H$ with $h$ vertices, there exists a constant $\mu>0$ such that for any sufficiently large $n$ with $n\in h\mathbb{N}$, if $G$ is an $n$-vertex graph with $\delta(G)\ge(1-\frac{1}{\chi^*(H)}+\alpha)n$ and $(G,\mathcal{F})$ is a $\mu n$-bounded incompatibility system, then there exists a compatible $H$-factor in $G$, where the value $\chi^*(H)$ is either the chromatic number $\chi(H)$ or the critical chromatic number $\chi_{cr}(H)$ and we provide a dichotomy as in the K\"{u}hn--Osthus result. Moreover, we give examples $H$ for which there exists an $\mu n$-bounded incompatibility system $(G, \mathcal{F})$ with $n\in h\mathbb{N}$ and $\delta(G)\ge(1-\frac{1}{\chi^*(H)}+\frac{\mu}{2})n$ such that $G$ contains no compatible $H$-factor. Unlike in the previous work of K\"{u}hn and Osthus on embedding $H$-factors, our proof uses the lattice-based absorption method.
\end{abstract}

\section{Introduction}
All the graphs considered here are finite, undirected and simple.
Let $H$ be an $h$-vertex graph and $G$ be an $n$-vertex graph. An \emph{$H$-tiling} is a collection of vertex-disjoint copies of $H$ in $G$. An \emph{$H$-factor} (or \emph{perfect $H$-tiling}) is an $H$-tiling which covers all vertices of $G$. Note that $n\in h\mathbb{N}$ is a necessary condition for $G$ containing an $H$-factor.

\subsection{Perfect graph tilings}

One of the most fundamental research topics in extremal graph theory is to determine sufficient conditions forcing spanning structures, such as perfect matchings, Hamilton cycles, $H$-factors, etc. Textbook results of Hall and Tutte give a sufficient condition for the existence of a perfect matching (see e.g. \cite{Diestel}). A classical theorem of Dirac \cite{Dirac} states that every graph $G$ on $n\ge 3$ vertices with minimum degree $\delta(G)\ge n/2$ contains a Hamilton cycle. For $H$-factors,
when $H=K_k$, the seminal Hajnal--Szemer\'edi theorem \cite{HS} characterizes the minimum degree forcing a $K_k$-factor and the bound is best possible.
\begin{theorem}[Hajnal and Szemer\'edi \cite{HS}]\label{thm:HS}
Let $G$ be an $n$-vertex graph with $n\in k\mathbb{N}$. If $\delta(G)\ge \left(1-1/k\right)n$, then $G$ contains a $K_k$-factor.
\end{theorem}

Up to an error term, the following theorem of Alon and Yuster \cite{AY} generalizes Theorem~\ref{thm:HS}.

\begin{theorem}[Alon and Yuster \cite{AY}]\label{thm:AY}
For every $\alpha>0$ and every graph $H$ there exists an integer $n_0$ such that every graph $G$ whose order $n\ge n_0$ is divisible by $|H|$ and whose minimum degree is at least $(1-1/\chi(H)+\alpha)n$ contains an $H$-factor.
\end{theorem}

Later, it was proved by Koml{\'o}s, S{\'a}rk{\"o}zy and Szemer{\'e}di~\cite{KSS} that the $o(n)$ term in Theorem \ref{thm:AY} can be replaced with a constant $C=C(H)$, which resolves a conjecture of Alon and Yuster~\cite{AY}.
As given in \cite{cooley2007perfect,Kawa}, there are graphs $H$ for which the term $1-1/\chi(H)$ in the minimum degree condition can be improved significantly. To illustrate this, Koml{\'o}s~\cite{kom00} introduced the concept of critical chromatic number. 
The \emph{critical chromatic number} $\chi_{cr}(H)$ of a graph $H$ is defined as $(\chi(H)-1)|H|/(|H|-\sigma(H))$, where $\sigma(H)$ denotes the minimum size of the smallest color class in a proper coloring of $H$ with $\chi(H)$ colors.
Note that $\chi(H)-1<\chi_{cr}(H)\le\chi(H)$.
Koml\'{o}s \cite{kom00} proved that one can replace $\chi(H)$ with $\chi_{cr}(H)$ at the price of obtaining an $H$-tiling covering all but $\eps n$ vertices, and he also conjectured that the error term $\eps n$ can be replaced with a constant that only depends on $H$. This was confirmed by Shokoufandeh and Zhao \cite{shokou03}.

\begin{theorem}[Shokoufandeh and Zhao \cite{shokou03}]\label{thm:ShZh}
For any $H$ there is an integer $n_0$ so that if $G$ is a graph on $n\ge n_0$ vertices and minimum degree at least $(1-1/\chi_{cr}(H))n$, then $G$ contains an $H$-tiling that covers all but at most $5|H|^2$ vertices.
\end{theorem}


Further, a significant result of K\"uhn and Osthus \cite{KO} managed to characterize, up to an additive constant, the best possible minimum degree condition which forces an $H$-factor. To state their result, we need more definitions here.
Let $\chi(H)=k$ and $\C(H)$ be the family of all proper $k$-colorings of $H$.
Given a $k$-coloring $c$ in $\C(H)$, let $h_1^c\le \cdots \le h_k^c$ be the sizes of the color classes of $c$ and $D(c):=\{h_{i+1}^c-h_i^c \mid i\in [k-1]\}$.
Let \[D(H):=\bigcup_{c\in\C(H)}D(c).\]
Then we denote by $\hcf_{\chi}(H)$ the highest common factor of all the integers in $D(H)$. In particular, if $D(H)=\{0\}$, then we set $\hcf_{\chi}(H):=\infty$.
Also, we write $\hcf_c(H)$ for the highest common factor of all the orders of components of $H$. 
If $\chi(H)> 2$, then define $\hcf(H)=1$ if $\hcf_{\chi}(H)=1$.
If $\chi(H)= 2$, then define $\hcf(H)=1$ if both $\hcf_{c}(H)=1$ and $\hcf_{\chi}(H)\le 2$.
Then let
\[
\chi^*(H) = \begin{cases}
\chi_{cr}(H) &\text{ if } \hcf(H)=1,\\
\chi(H) &\text{ otherwise.}
\end{cases}
\]

\begin{theorem}[K\"{u}hn and Osthus \cite{KO}]\label{thm:KO}
For any $H$ there exist integers $C=C(H)$ and $n_0=n_0(H)$ such that every graph $G$ whose order $n\ge n_0$ is divisible by $|H|$ and whose minimum degree is at least $(1-1/\chi^*(H))n+C$ contains an $H$-factor.
\end{theorem}
The degree condition in Theorem~\ref{thm:KO} is best possible up to the constant $C$ and there are also graphs $H$ for which the constant $C$ cannot be omitted entirely \cite{KO}. There are some generalisations of Theorem~\ref{thm:KO}. For example, K\"{u}hn, Osthus and
Treglown~\cite{oretypefactor} determined, asymptotically, the Ore-type degree condition
forcing an $H$-factor; Hyde and Treglown~\cite{degreesequencefactor} proved a degree sequence version of Theorem~\ref{thm:KO}. Very recently, Hurley, Joos and Lang~\cite{Hurley2022SufficientCF} recovered and extended the work of K\"{u}hn and Osthus by replacing the $H$-factor with bounded degree graphs with components of sublinear order.
\subsection{Motivation}

As mentioned above, many typical Dirac-type results in extremal graph theory are of the form ``under certain degree conditions,
$G$ has property $\mathcal{P}$''. Once such a result is established, it is natural to ask how
strongly does $G$ possess $\mathcal{P}$? In particular, this has motivated recent trends in the
study of robustness of graph properties (see~\cite{Sudakov}), aiming to strengthen classical results in
extremal graph theory and probabilistic combinatorics. What measures
of robustness can one utilize? In this direction, several different measures of robustness have been explored in
\begin{enumerate}
  \item (sparse) random graphs $G(n, p)$ in the setting of so-called resilience (see e.g. \cite{AS,BCS,BKS,DKMS,KLS0,KLS01,SV});
  \item edge-colored graphs (see e.g. \cite{AG,AJMP,BE,CD,CKPY,CP,CP1,ES,rainbowblowup,KKKO20,Lo,MPS,Shearer,Yuster});
  \item Maker-Breaker games (see e.g. \cite{Beck,BP,CE,Kri,KSSS}).
\end{enumerate}
For further works in this area the reader is referred to a comprehensive survey by Sudakov~\cite{Sudakov}.

Finding spanning subgraphs in edge-colored graphs with certain constraints has also been widely and well studied. An edge-coloring $c$ of a graph is \emph{(globally) $k$-bounded} if every color appears at most $k$ times in the coloring, while $c$ is \emph{locally $k$-bounded} if every color appears at most $k$ times at any given vertex (so $c$ is a proper edge-coloring when $k=1$). A fundamental line of research is to find \emph{rainbow subgraphs} (all edges have distinct colors) in $g$-bounded edge-colorings of graphs and \emph{properly colored subgraphs} (any adjacent edges receive different colors) in locally $\ell$-bounded edge-colorings of graphs. For perfect matchings, Erd\H{o}s and Spencer \cite{ES} proved that any $\frac{n-1}{16}$-bounded edge-coloring of $K_{n,n}$ admits a rainbow perfect matching. Recently, Coulson and Perarnau \cite{CP} considered the sparse version and obtained that there exist $\mu>0$ and $n_0\in \mathbb{N}$ such that if $n\ge n_0$ and $G$ is a \emph{Dirac bipartite graph} on $2n$ vertices, which is a balanced bipartite graph with minimum degree at least a quarter of its order, then any $\mu n$-bounded edge-coloring of $G$ contains a rainbow perfect matching. For Hamilton cycles, a conjecture of Bollob\'as and Erd\H{o}s \cite{BE} states that every locally $\left(\lfloor\frac{n}{2}\rfloor-1\right)$-bounded edge-colored $K_n$ contains a properly colored Hamilton cycle. There is a series of partial results toward this conjecture (see e.g. \cite{AG,CD,Shearer}). In \cite{Lo}, Lo proved that the Bollob\'as--Erd\H{o}s conjecture is true asymptotically. For sparse version, Coulson and Perarnau \cite{CP1} derived that there exist $\mu>0$ and $n_0\in \mathbb{N}$ such that if $n\ge n_0$ and $G$ is a \emph{Dirac graph} on $n$ vertices, which is a graph with minimum degree at least half of its order, then any $\mu n$-bounded edge-coloring of $G$ contains a rainbow Hamilton cycle.
For $H$-factors, Coulson, Keevash, Perarnau and Yepremyan \cite{CKPY} showed that there is a constant $\mu$ such that any $\mu n$-bounded edge-coloring of $G$ with $\delta(G)\ge (1-\frac{1}{\chi^*(H)}+o(1))n$ contains a rainbow $H$-factor.

In this paper, we consider a more general setting of incompatibility systems, which was first proposed by Krivelevich, Lee and Sudakov~\cite{KLS}.
\begin{defn}[Incompatibility system]\label{is}
Let $G=(V,E)$ be a graph.
\begin{itemize}
  \item An \emph{incompatibility system} $\mathcal{F}$ over $G$ is a family $\mathcal{F}=\{F_v\}_{v\in V}$ such that for every $v\in V$, $F_v$ is a family of 2-subsets of edges incident with $v$, i.e. $F_v\subseteq \{\{e,e'\}\in {E\choose 2}: e\cap e'=\{v\}\}$. For simplicity, we often denote the system as $(G,\mathcal{F})$.
  \item For any two edges $e$ and $e'$, if there exists some vertex $v$ such that $\{e,e'\}\in F_v$, then we say that $e$ and $e'$ are \emph{incompatible} at $v$. Otherwise, they are \emph{compatible}. A subgraph $H\subseteq G$ is \emph{compatible} if all pairs of edges are compatible.
  \item For a positive integer $\Delta$, an incompatibility system $(G,\mathcal{F})$ is \emph{$\Delta$-bounded} if for any vertex $v$ and any edge $e$ incident with $v$, there are at most $\Delta$ other edges incident with $v$ that are incompatible with $e$.
\end{itemize}
\end{defn}
Locally $\ell$-bounded edge-colorings can be viewed as transitive $(\ell-1)$-bounded incompatibility systems in which two adjacent edges are incompatible if only they have the same color, and compatible subgraphs generalize the concept of properly colored subgraphs. Note that there is also a similar generalization of $g$-bounded edge-colorings, called \emph{systems of conflicts} (see \cite{CP}). Definition \ref{is} was firstly introduced in \cite{KLS} as a measure of robustness and motivated by two concepts in graph theory. First, it generalizes \emph{transition systems} introduced by Kotzig \cite{Kotzig} in 1968. In our
terminology, a transition system is simply a 1-bounded incompatibility
system. Kotzig's work was motivated by a problem of Nash--Williams on cycle covering of Eulerian graphs (see, e.g. Section 8.7 in \cite{Bondy}).

In \cite{KLS}, Krivelevich, Lee and Sudakov studied the robustness of Hamiltonicity of Dirac graphs with respect to the incompatibility system and derived Theorem \ref{comp H-C}. 

\begin{theorem}[\cite{KLS}]\label{comp H-C}
There exists a constant $\mu>0$ such that the following holds for large enough $n$. For every $n$-vertex Dirac graph $G$ and a $\mu n$-bounded incompatibility system $\mathcal{F}$ defined over $G$, there exists a compatible Hamilton cycle.
\end{theorem}
Their proof is based on P\'osa's rotation-extension technique. Theorem \ref{comp H-C} settled in a very strong form, a conjecture of H\"aggkvist from 1988 (see Conjecture 8.40 in \cite{Bondy}). They further studied compatible Hamilton cycles in random graphs in \cite{KLS1} and proved that there is a constant $\mu > 0$ such that if $p\gg \frac{\log n}{n}$, then w.h.p. $G=G(n, p)$ contains a compatible Hamilton cycle for every $\mu np$-bounded incompatibility system defined over $G$, which strengthens the result about Hamilton cycles in $G(n,p)$ without restrictions of incompatibility systems.

The concept of incompatibility system appears to provide a new and interesting take on robustness
of graph properties. Krivelevich, Lee and Sudakov \cite{KLS} also suggested looking at how various extremal results can be strengthened in this context.

\subsection{Main result and discussion}
It is natural to consider other compatible spanning structures. We study minimum degree conditions for $H$-factors in incompatibility system. Note that throughout this paper, we always assume $\chi(H)\ge 2$, since $\chi(H)=1$ is a trivial case. For convenience, given constants $\mu,\delta$ and $n\in \mathbb{N}$, an \emph{$(n,\delta,\mu)$-incompatibility system }$(G,\mathcal{F})$ consists of an $n$-vertex graph $G$ with $\delta(G)\ge \delta n$ and a $\mu n$-bounded incompatibility system $\mathcal{F}$ over $G$. Our main result is formally stated as follows.

\begin{theorem}\label{main thm}
Let $h\in \mathbb{N}$ and $H$ be any $h$-vertex graph. For any $\alpha>0$, there exists a constant $\mu>0$ such that for any sufficiently large $n$ with $n\in h\mathbb{N}$, every $(n,1-\frac{1}{\chi^*(H)}+\alpha,\mu)$-incompatibility system $(G,\mathcal{F})$ contains a compatible $H$-factor. In particular, the term of $\alpha$ in the minimum degree condition cannot be omitted when $H$ is a complete $r$-partite graph for $r\ge 3$.
\end{theorem}
Towards the `in particular' part in the statement, an obvious question is for which graphs the error term $\alpha n$ in the minimum degree condition can be replaced by a constant term.
\subsubsection{A space barrier}
For $r\ge3$ and $h_1,\ldots,h_r\ge 1$,
we use $K_r(h_1,h_2,\ldots,h_r)$ to denote the complete $r$-partite graph with each part of size $h_i$. In particular, when $h_i=s$ for any $i\in [r]$, we write $K_r^s$ for $K_r(s,s,\ldots,s)$. 
Now we show that the minimum degree condition is in some sense tight by the following result.
\begin{prop}\label{extremal graph}
Let $H=K_r(h_1,h_2,\ldots,h_r)$, $h=\sum_{i\in[r]} h_i$ and $n\in h\mathbb{N}$. For any $0<\mu< \frac{{\chi_{cr}(H)}-r+1}{\chi_{cr}(H)}-\frac{r-1}{n}$, there exists an $(n,1-\frac{1}{\chi^*(H)}+\frac{\mu}{2},\mu)$-incompatibility system $(G,\mathcal{F})$ which contains no compatible $K_r(h_1,h_2,\ldots,h_r)$-factor.
\end{prop}
To prove Proposition \ref{extremal graph}, we need two results as follows which provide lower bound constructions for the minimum degree threshold forcing an $H$-factor.
\begin{prop}\emph{\cite{kom00}}\label{kom}
For every graph $H$ with $2\le\chi(H)=:r$ and every integer $n$ that is divisible by
$|H|$, there exists a complete $r$-partite graph $G$ of order $n$ with minimum degree $(1-\frac{1}{\chi_{cr}(H)})n-1$ which does not contain an $H$-factor.
\end{prop}
Since the above $n$-vertex graph $G$ is complete $r$-partite with minimum degree $(1-\frac{1}{\chi_{cr}(H)})n-1$, the size of the largest vertex class of $G$ is $\frac{n}{{\chi_{cr}(H)}}+1$, while the size of the smallest vertex class is at least $n-(r-1)(\frac{n}{{\chi_{cr}(H)}}+1)=\frac{{\chi_{cr}(H)}-r+1}{\chi_{cr}(H)}n-r+1$.
\begin{prop}\emph{\cite{KO}}\label{ko}
Let $H$ be a graph with $3\le \chi(H)=:r$ and let $h\in \mathbb{N}$. Let $G$
be the complete $r$-partite graph of order $k|H|$ whose vertex classes $V_1,\ldots,V_r$
satisfy $|V_1|=
\lfloor k|H|/r\rfloor+1, |V_2|=\lceil k|H|/r\rceil-1$ and
$\lfloor k|H|/r\rfloor\le|V_i|\le\lceil k|H|/r\rceil$
for all $i\ge3$. So $\delta(G)=\lceil(1-\frac{1}{\chi(H)})|G|\rceil-1$. If $\hcf(H)\neq1$, then $G$ does not contain an $H$-factor.
\end{prop}
\begin{proof}[\textbf{Proof of Proposition \ref{extremal graph}}]
Let $G_0$ be an $n$-vertex complete $r$-partite graph with $\delta(G_0)\ge(1-\frac{1}{\chi^*(H)})n-1$ such that $G_0$ contains no $H$-factors. Indeed such $G_0$ can be obtained from Proposition~\ref{kom} or Proposition~\ref{ko}. Write $V_1,\ldots,V_r$ for the $r$ parts of $G_0$ and inside every part $V_i$ of $G_0$, we add a spanning bipartite subgraph with minimum degree at least $\frac{\mu n}{2}+1$ and maximum degree at most $\mu n$. Denote the resulting graph by $G$. Hence, $\delta(G)\ge \left(1-\frac{1}{\chi^*(H)}+\frac{\mu}{2}\right)n$ and for every $i\in[r]$, $G[V_i]$ is a triangle-free graph with $\delta(G[V_i])\ge \frac{\mu n}{2}+1$ and $\Delta(G[V_i])\le \mu n$.
Now we define an incompatibility system $\mathcal{F}$ over $G$. For any two different parts $V_i,V_j$ of $G$, let $v$ be any vertex in $V_i$ and $u,w$ be any two different vertices in $V_j$. If $uw$ is an edge in $G[V_j]$, then let $vu$ and $vw$ be incompatible at $v$. Since $\Delta(G[V_j])\le \mu n$, $\mathcal{F}$ is $\mu n$-bounded. 

Now it remains to verify that there is no compatible $K_r(h_1,h_2,\ldots,h_r)$-factor. Before that, for any subgraph $F\subseteq G$ we define the \emph{index vector} $\textbf{i}_{F}=(x_1,\ldots,x_r)$ by choosing $x_i=|V(F)\cap V_i|, i\in[r]$. We claim that for every compatible copy of $K_r(h_1,h_2,\ldots,h_r)$ in $G$, its index vector is an $r$-tuple of $h_1,h_2,\ldots,h_r$. In fact, assume that there is a compatible copy of $K_r(h_1,h_2,\ldots,h_r)$ in $G$, denoted by $K$, whose index vector is not a permutation of $h_1,h_2,\ldots,h_r$. Then $K$ has an edge in $G[V_i]$ for some $i\in [r]$, say $uw$. Since $r\ge 3$, $uw$ is contained in a triangle of $K$, say $vuw$. As $G[V_i]$ is triangle-free, one can observe that $v$ is in some $V_j$ for $j\neq i$ and by assumption $vu$ and $vw$ are compatible at $v$, which contradicts the definition of $\mathcal{F}$. Hence, since $G_0$ contains no $K_r(h_1,h_2,\ldots,h_r)$-factor, it follows from the aforementioned claim that $G$ contains no compatible $K_r(h_1,h_2,\ldots,h_r)$-factor as well.
\end{proof}


The rest of the paper is organised as follows. In Section~\ref{sec2}, we set up some basic notation and outline the proof of Theorem \ref{main thm}. Then we state two crucial results (Lemmas \ref{abs} and \ref{almost}) that decode the proof of Theorem \ref{main thm}. Sections \ref{sec3} and \ref{sec4} are devoted to proving Lemmas \ref{abs} and \ref{almost}, respectively.
\section{Notation and preliminaries}\label{sec2}
For a graph $G$, we use $e(G)$ to denote the number of edges of $G$.
For a graph $G=(V,E)$ and $k$ pairwise disjoint vertex subsets $U_1,\ldots, U_k\subseteq V$, we use $G[U_1,\ldots, U_k]$ to denote the $r$-partite subgraph induced by $U_1,\ldots, U_r$.
In the proofs, if we choose $a\ll b$, then this means that for any $b>0$, there exists $a_0>0$ such that for any $a<a_0$ the subsequent statement holds. Hierarchies of other lengths are defined similarly.

\subsection{Proof strategy and main tools}\label{sec2.1}
Our proof uses the absorption method, pioneered by the work of R\"odl, Ruci\'nski and Szemer\'edi~\cite{RRS} on perfect matchings in hypergraphs. In recent years, the method has become an extremely important tool for studying the existence of spanning structures in graphs, digraphs and hypergraphs. Before its appearance a well-known systematic way in this area is the blow-up lemma due to Koml\'os, S\'ark\"ozy and Szemer\'edi~\cite{KSSblowup}, which has successfully been adopted by many scholars for various classical results \cite{AY,KO}. It is worth noting that Glock and Joos \cite{rainbowblowup} developed a rainbow version of the blow-up lemma, which enables the systematic study of rainbow embeddings of bounded degree spanning subgraphs~\cite{joos_2020,KKKO20}. However, there has been no such variant for incompatibility systems. Instead, we utilize the absorption framework and develop a different counting argument for embedding compatible graphs under certain pseudorandomness conditions (see Lemma~\ref{K_h^r}).

The general idea of absorption is to split the problem of finding perfect tilings into two subproblems. The first major task is to define and find an absorbing set in the host graph which can `absorb' left-over vertices.
We will start with the notion of absorbers needed in our proof.
\begin{defn}[Absorber]\label{defabs}
Let $H$ be an $h$-vertex graph, $G$ be an $n$-vertex graph and $\mathcal{F}$ be an incompatibility system over $G$.
For any $h$-set $S\subseteq V(G)$ and integer $t$, we say that a set $A_S\subseteq V(G)\setminus S$ is an \emph{$(H,t)$-absorber} for $S$ if $|A_S|\le h^2t$ and both $G[A_S]$ and $G[A_S\cup S]$ contain compatible $H$-factors.
\end{defn}

The first task is handled as follows.
\begin{defn}[Absorbing set]
Let $H$ be an $h$-vertex graph, $G$ be an $n$-vertex graph, $\mathcal{F}$ be an incompatibility system over $G$ and $\xi$ be a constant.
A set $A\subseteq V(G)$ is called a $\xi$-\emph{absorbing set} (for $V(G)$) if for any set $R\subseteq V(G)\setminus A$ with $|R|\le \xi n$ and $|A\cup R|\in h\mathbb{N}$, the graph $G[A\cup R]$ contains a compatible $H$-factor.
\end{defn}

\begin{lemma}[Absorbing lemma]\label{abs}
Let $H$ be an $h$-vertex graph with $\chi(H)\ge 2$. For any $\alpha,\sigma>0$, there exist $\mu,\xi>0$ such that for any sufficiently large $n$, if $(G,\mathcal{F})$ is an $\left(n,1-\frac{1}{\chi^*(H)}+\alpha,\mu\right)$-incompatibility system, then $G$ contains a $\xi$-absorbing set $A$ of size at most $\sigma n$.
\end{lemma}

The second major task in  absorption arguments for perfect tilings is to find a tiling that covers most of the vertices, leaving just a small portion of vertices uncovered. We will sometimes call such a tiling an \emph{almost perfect tiling} or an \emph{almost cover}.

\begin{lemma}[Almost cover]\label{almost}
Let $H$ be an $h$-vertex graph with $\chi(H)\ge 2$. For any $\alpha,\tau>0$, there exists $\mu>0$ such that for any sufficiently large $n$, if $(G,\mathcal{F})$ is an $\left(n,1-\frac{1}{\chi_{cr}(H)}+\alpha,\mu\right)$-incompatibility system, then there exists a compatible $H$-tiling covering all but at most $\tau n$ vertices of $G$.
\end{lemma}

Now we are ready to prove Theorem~\ref{main thm} to end this section.
\begin{proof}[\textbf{Proof of Theorem~\ref{main thm}}]
Given $\alpha>0$, we choose
$$\frac{1}{n}\ll\mu\ll\tau\ll\xi\ll\sigma\ll\alpha.$$
Applying Lemma \ref{abs} to $G$ with $\sigma\le \frac{\alpha}{2}$, we obtain a $\xi$-absorbing set $A$ of size at most $\sigma n$. Since $\delta(G-A)\ge \left(1-\frac{1}{\chi^*(H)}+\frac{\alpha}{2}\right)n$, we can apply Lemma \ref{almost} to $G-A$ to obtain a compatible $H$-tiling $\mathcal{H}_1$ covering all but at most $\tau n$ vertices in $G-A$. Denote by $R$ the set of uncovered vertices in $G-A$. Since $\tau\ll\xi$, $G[A\cup R]$ contains a compatible $H$-factor $\mathcal{H}_2$. Thus, $\mathcal{H}_1\cup \mathcal{H}_2$ is a compatible $H$-factor of $G$.
\end{proof}


\section{Almost compatible $H$-factor}\label{sec3}
The main tools in the proof of Lemma \ref{almost} are Szemer\'edi's Regularity Lemma and an embedding lemma (see Lemma~\ref{K_h^r}). We first give some definitions and lemmas.
\subsection{Regularity}\label{sec3.1}
\begin{defn}[Regular pair]
Given a graph $G$ and disjoint vertex subsets $X,Y\subseteq V(G)$, the \emph{density} of the pair $(X,Y)$ is defined as $$d(X,Y):=\frac{e(X,Y)}{|X||Y|},$$ where $e(X,Y):=e(G[X,Y])$. For $\varepsilon>0$, the pair $(X,Y)$ is $\varepsilon$-\emph{regular} if for any $A\subseteq X, B\subseteq Y$ with $|A|\ge \varepsilon |X|, |B|\ge \varepsilon |Y|$, we have $$|d(A,B)-d(X,Y)|<\varepsilon.$$ Additionally, if $d(X,Y)\ge d$ for some $d\ge0$, then we say that $(X,Y)$ is $(\varepsilon,d)$-\emph{regular}.
\end{defn}


\begin{fact}\label{large deg}
Let $(X,Y)$ be an $(\varepsilon,d)$-regular pair, and $B\subseteq Y$ with $|B|\ge \varepsilon |Y|$. Then all but $\varepsilon |X|$ vertices in $X$ have degree at least $(d- \varepsilon)|B|$ in $B$.
\end{fact}

\begin{fact}[Slicing lemma, \cite{KS}]\label{slicing lem}
Let $(X,Y)$ be an $(\varepsilon,d)$-regular pair, and for some $\eta>\varepsilon$, let $X'\subseteq X, Y'\subseteq Y$ with $|X'|\ge \eta |X|, |Y'|\ge \eta |Y|$. Then $(X',Y')$ is an $\varepsilon'$-regular pair with $\varepsilon'=\max\{\varepsilon/\eta,2\varepsilon\}$, and for its density $d'$ we have $|d'-d|<\varepsilon$.
\end{fact}

\begin{lemma}[Degree form of the Regularity Lemma, \cite{KS}]\label{reg lem}
For every $\varepsilon>0$, there is an $M=M(\varepsilon)$ such that if $G=(V,E)$ is any graph and $d\in[0,1]$ is any real number, then there is a partition $V=V_0\cup V_1\cup \ldots \cup V_k$ and a spanning subgraph $G'\subseteq G$ with the following properties:
\begin{itemize}
  \item $1/\varepsilon \le k\le M$,
  \item $|V_0|\le \varepsilon |V|$,
  \item $|V_i|=m$ for all $1\le i\le k$ with $m\le \varepsilon |V|$,
  \item $d_{G'}(v)>d_{G}(v)-(d+\varepsilon)|V|$ for all $v\in V$,
  \item $e(G'[V_i])=0$ for all $i\ge 1$,
  \item all pairs $(V_i,V_j)$ $(1\le i<j\le k)$ are $\varepsilon$-regular in $G'$ with density $0$ or at least $d$.
\end{itemize}
\end{lemma}
\noindent Moreover, we usually call $V_0,V_1,\ldots,V_k$ \emph{clusters} and call the cluster $V_0$ \emph{exceptional set}.

\begin{defn}[Reduced graph]\label{reduced graph}
Given an arbitrary graph $G=(V,E)$, a partition $V=V_1\cup \ldots \cup V_k$, and two parameters $\varepsilon,d>0$, the \emph{reduced graph} $R=R(\varepsilon,d)$ of $G$ is defined as follows:
\begin{itemize}
  \item $V(R)=[k]$,
  \item $ij\in E(R)$ if and only if $(V_i,V_j)$ is $(\varepsilon,d)$-regular.
\end{itemize}
\end{defn}

As remarked in \cite{KS}, a typical application of degree form of the Regularity Lemma begins with a graph $G=(V,E)$ and appropriate parameters $\varepsilon,d>0$, and then obtains a partition $V=V_0\cup V_1\cup \ldots \cup V_k$ and a subgraph $G'$ with above-mentioned properties. Then we usually drop the exceptional set $V_0$ to get a pure graph $G'-V_0$ and study the properties of reduced graph $R=R(\varepsilon,d)$. By Lemma \ref{reg lem},
\begin{align*}
   \delta(R)\ge \frac{\delta(G)-(d+\varepsilon)|V|-|V_0|}{m} \ge \frac{\delta(G)-(d+2\varepsilon)|V|}{m}.
\end{align*}
In particular, if $\delta(G)\ge c|V|$, then $\delta(R)\ge (c-d-2\varepsilon)k$.


\subsection{Proof of Lemma~\ref{almost}}\label{sec3.2}
To find an almost cover, we adopt the approach of Koml\'{o}s in \cite{kom00}, making use of Szemer\'{e}di's Regularity Lemma and a result of Shokoufandeh and Zhao~\cite{shokou03} (see Theorem~\ref{thm:ShZh}). In \cite{kom00}, due to the graph counting lemma (e.g, see \cite{KS}), the arguments often reduce to greedily picking vertex-disjoint copies of $H$ within a bunch of clusters under certain pseudorandom properties. To adopt this approach in the context of incompatibility systems which boils down to embedding compatible copies of $H$, we develop a `compatible' variant of the graph counting lemma as follows.
\begin{lemma}\label{K_h^r}
For constant $\eta,d>0$ and positive integers $r,h_1,\ldots,h_r$ with $\sum_{i=1}^rh_i=:h$, there exist positive constants $\varepsilon^\ast=\varepsilon^\ast(r,d,h), c=c(r,d,h)$ and $\mu=\mu(r,d,h,\eta)$ such that the following holds for sufficiently large $n$. Let $(G,\mathcal{F})$ be a $\mu n$-bounded incompatibility system with $|G|=n$ and $U_1,\ldots,U_r$ be pairwise vertex-disjoint sets in $V(G)$ with $|U_i|\ge \eta n$, $i \in [r]$ and every pair $(U_i,U_j)$ being $(\varepsilon^\ast,d)$-regular. Then there exist at least $c\prod_{i=1}^r|U_i|^{h_i}$ compatible copies of $K_r(h_1,\ldots,h_r)$ in $G[U_1,\ldots,U_r]$, each containing exactly $h_i$ vertices in $U_i$ for every $i\in[r]$.
\end{lemma}
We also give a notion of bottle-graphs by Koml\'{o}s in \cite{kom00}.
\begin{defn}
Given integers $h,r$ and an $h$-vertex graph $H$ with $\chi(H)=r$, we define a \emph{bottle-graph} $B$ of $H$ as a complete $r$-partite graph with part sizes $(r-1)\sigma(H),h-\sigma(H),\ldots,h-\sigma(H)$.
\end{defn}
It is easy to see that $\chi_{cr}(B)=\chi_{cr}(H)$ and $B$ contains an $H$-factor consisting of $r-1$ copies of $H$. Based on this, our proof strategy is to find an almost cover with compatible copies of $B$.


\begin{proof}[Proof of Lemma \ref{almost}]
Given $\alpha,\tau>0$ and an $h$-vertex graph $H$, we let $\chi(H)=:r$ and choose $$\frac{1}{n}\ll \mu\ll\varepsilon\ll \rho\ll\alpha,\tau,\frac{1}{r},\frac{1}{h}.$$ We first apply Lemma \ref{reg lem} to $G$ with $d=\rho$ to obtain a spanning subgraph $G'$ and a partition $V(G)=V_0\cup V_1\cup\ldots \cup V_k$ for some $1/\varepsilon \le k\le M$ with $|V_i|=m\ge \frac{(1-\varepsilon)n}{k}$ for every $i\in[k]$. Let $R=R(\varepsilon,\rho)$ be the reduced graph for this partition. Since $\varepsilon\ll \rho \ll\alpha$
, we have $\delta(R)\ge \left(1-\frac{1}{\chi_{cr}(H)}+\alpha-\rho-2\varepsilon\right)k\ge \left(1-\frac{1}{\chi_{cr}(H)}+\frac{\alpha}{2}\right)k$.
Since $\chi_{cr}(B)=\chi_{cr}(H)$, by Theorem~\ref{thm:ShZh}, $R$ has a $B$-tiling  $\mathcal{B}$ that covers all but at most $5|B|^2$ vertices in $R$. 

Recall that $B$ has part sizes $(r-1)\sigma(H),h-\sigma(H),\ldots,h-\sigma(H)$. Here we write $\ell:=h-\sigma(H), y:=(r-1)\sigma(H)$ and $r'=|B|=(r-1)h$ for simplicity, and thus $r'=y+(r-1)\ell$.
Given a copy of $B$ in $\mathcal{B}$, without loss of generality, we may assume that its vertex set is $\{V_{1},\ldots, V_{r'}\}$ with the $r$-partition of $V(B)$ denoted as \[\mathcal{W}_1=\{V_1,\ldots,V_y\}~ \text{and} ~ \mathcal{W}_{s+1}=\{V_{y+1+(s-1)\ell},\ldots,V_{y+s\ell}\}~\text{for}~s\in [r-1].\] Note that every pair of clusters $V_i,V_j$ from distinct parts forms an $(\eps,\rho)$-regular pair.

We shall greedily embed in $G'$ vertex-disjoint compatible copies of $B$ that cover almost all the vertices in $\cup_{i=1}^{r'}V_i$.
Now for every $i\in[y]$ we divide $V_{i}$ arbitrarily into $\ell$ subclusters $V_{i,1},\ldots,V_{i,\ell}$ of (almost) equal size. For every $j\in [y+1,r']$ we divide $V_{j}$ into $y$ subclusters $V_{j,1},\ldots,V_{j,y}$ of (almost) equal size. Here for simplicity we may further assume that $|V_{i,i'}|=\frac{m}{\ell}$ for $i\in[y],i'\in[\ell]$ and $|V_{j,j'}|=\frac{m}{y}$ for every $j\in[y+1,r'],j'\in[y]$.
We call a family $\{V_{i_s,j_s}\}_{s=1}^{r}$ of $r$ subclusters \emph{legal} if $V_{i_s}\in \mathcal{W}_s$ for every $s\in[r]$, i.e., $\{V_{i_s}\}_{s=1}^{r}$ forms a copy of $K_{r}$ in $R$.
Note that every $\mathcal{W}_s$ ($s\in [r]$) contains exactly $y\ell$ subclusters in total. Therefore we can greedily partition the set of all subclusters into $y\ell$ pairwise disjoint legal families.

Now if for every legal family $\{V_{i_s,j_s}\}_{s=1}^{r}$ we have a $B$-tiling in $G'[\bigcup_{s=1}^{r} V_{i_s,j_s}]$ that covers all but at most $\frac{\tau}{4y\ell}r' m$ vertices of $\bigcup_{s=1}^{r} V_{i_s,j_s}$, then putting them together would give a $B$-tiling covering all but at most $\frac{\tau}{4}r'm$ vertices of $V_1\cup V_2\cup\cdots\cup V_{r'}$. Applying this to every copy of $B$ from $\mathcal{B}$ would give us a $B$-tiling covering all but at most
\[
|V_0|+5|B|^2 m+ |\mathcal{B}|\frac{\tau}{4}r'm< \eps  n+ 5(r')^2\eps n+ \frac{\tau}{4}n< \tau n
\]
vertices in $G$. So to complete the proof of Lemma \ref{almost}, it is sufficient to prove the following claim.

\begin{claim}\label{embedding lem}
Given any legal family $\{V_{i_s,j_s}\}_{s=1}^{r}$, $G[V_{i_1,j_1},\ldots,V_{i_r,j_r}]$ admits a $B$-tiling covering all but at most $\frac{\tau}{4y\ell} r'm$ vertices of $\bigcup_{s=1}^{r} V_{i_s,j_s}$.
\end{claim}

For convenience,
we write $Y_s:=V_{i_{s},j_{s}}$ with $s\in [r]$. Recall that $|Y_1|=\frac{m}{\ell}$ and $|Y_s|=\frac{m}{y}$ for $s\in[2,r]$. 
Now it suffices to show that for any $Y_s'\subseteq Y_s$ with $s\in[r]$, each of size at least $\frac{\tau}{4y\ell} m$, there exists a compatible copy of $B$ with exactly $y$ vertices inside $Y_1'$ and $\ell$ vertices inside every $Y_s'$, $s\in[2,r]$.

For any distinct $s,t\in[r]$, the pair $(V_{i_s},V_{i_t})$ is $\eps$-regular with density at least $\rho$. Then Fact~\ref{slicing lem} implies that every two sets from $Y_1',\ldots,Y_{r}'$ forms an $\eps'$-regular pair with density at least $\rho-\eps$, where $\eps'=\frac{4y\ell}{\tau}\eps$.
Therefore by the fact that $\frac{1}{n}\ll\mu\ll\varepsilon\ll \rho,\tau,\frac{1}{r},\frac{1}{h}$, Lemma~\ref{K_h^r} applied to $G$ with $U_i=Y_i', d=\rho/2,\eps^*=\eps', \eta=\frac{\tau}{4y\ell} \frac{1-\eps}{k},h_1=y,h_i=\ell$ for $i\in[2,r]$, gives a desired compatible copy of $B$. This completes the proof of Claim~\ref{embedding lem}.
\end{proof}

\subsection{Proof of Lemma~\ref{K_h^r}: counting compatible subgraphs}\label{sec3.3}
In the rest of this section, we will focus on the proof of Lemma \ref{K_h^r}. Before that, we give a definition.
\begin{defn}[Good pair]\label{good pair}
Given an incompatibility system $(G,\mathcal{F})$, a vertex $v\in V(G)$ and a subgraph $H\subseteq G$, we say that $(v, H)$ is a \emph{good pair} if
      \begin{itemize}
        \item $H$ is compatible;
        \item $v$ is adjacent to all vertices of $H$;
        \item all edges from $v$ to $V(H)$ are mutually compatible at $v$.
      \end{itemize}
Moreover, the vertex $v$ is called the \emph{center} of the good pair $(v, H)$.
\end{defn}
We also state an obvious fact as follows, which is frequently used throughout the proofs.
\begin{fact}\label{fact1}
 Let $(G,\mathcal{F})$ be a $\mu n$-bounded incompatibility system with $|G|=n$. Then for every $v\in V(G)$, there are at most $2\mu n^2$ pairs $\{v_1,v_2\}$ of vertices from $N(v)$ that satisfy either of the following properties:
 \begin{enumerate}
   \item [(1)] $vv_1$ and $vv_2$ are incompatible at $v$;
   \item [(2)] $vv_1$ and $v_1v_2$ are incompatible at $v_1$.
 \end{enumerate}
\end{fact}

\begin{proof}[Proof of Lemma \ref{K_h^r}]
We prove this lemma by induction on $r$. Given $\eta,d>0$, we choose \[\frac{1}{n}\ll\varepsilon_r\ll \varepsilon_{r-1}\ll\cdots\ll\varepsilon_1\ll d,\frac{1}{r},\frac{1}{h}~\text{and additionally}~\frac{1}{n}\ll\mu\ll \eta, d,\frac{1}{r},\frac{1}{h}.\]
The base case $r=1$ is trivial since we can find $|U_1|\choose h_1$ $h_1$-subsets in $U_1$.
Suppose that Lemma \ref{K_h^r} is true for $r-1$. Now we consider the case $r$, that is, $U_1,\ldots,U_r$ are pairwise $(\varepsilon_r,d)$-regular and $|U_i|\ge \eta n$ for any $i \in [r]$. We want to find $c\prod_{i=1}^r |U_i|^{h_i}$ compatible copies of $K_r(h_1,\ldots,h_r)$ in $G[U_1,\ldots,U_r]$ for some $c=c(r,d,h)>0$.

We choose a subset $U_r'\subseteq U_r$ such that every vertex in $U_r'$ has degree at least $(d-\varepsilon_r)|U_i|$ in $U_i$ for every $i\in [r-1]$.
By Fact \ref{large deg}, there are at least $(1-\varepsilon_r)|U_r|$ vertices of $U_r$ with degree at least $(d-\varepsilon_r)|U_i|$ in $U_i$ for any fixed $i\in [r-1]$.
Hence, $|U_r'|\ge (1-(r-1)\varepsilon_r)|U_r|$.

Let $v$ be an arbitrary vertex in $U_r'$. Denote by $U_i'$ the neighborhood of $v$ in $U_i$ for every $i\in [r-1]$.
Then $|U_i'|\ge (d-\varepsilon_r)|U_i|$.
By Fact \ref{slicing lem}, for any $1\le i<j\le r-1$, $U_i'$ and $U_j'$ is an $(\varepsilon',d')$-regular pair with $\varepsilon'=\max\{\varepsilon_r/(d-\varepsilon_r),2\varepsilon_r\}$ and $|d'-d|<\varepsilon_r$, which is also an $(\varepsilon_{r-1},\frac{d}{2})$-regular pair by the fact that $\varepsilon_{r}\ll \varepsilon_{r-1}$ and $d'>d-\varepsilon_r\ge \frac{d}{2}$.
By the induction hypothesis on $r-1$ with $h'=\sum_{i=1}^{r-1}h_i$, there is a family $\mathcal{K}_v$ of at least $c(r-1,\frac{d}{2}, h')\prod_{i=1}^{r-1} |U_i'|^{h_i}$ compatible copies of $K_{r-1}(h_1,\ldots,h_{r-1})$ in $G[U_1',\ldots, U_{r-1}']$.
 In $\mathcal{K}_v$, we know from Fact~\ref{fact1} that there are at most $2\mu n^2\cdot n^{\sum_{i=1}^{r-1} h_i-2}=2\mu n^{\sum_{i=1}^{r-1} h_i}$ compatible copies of $K_{r-1}(h_1,\ldots,h_{r-1})$, each containing two vertices $v_1,v_2$ such that $vv_1$ and $vv_2$ are incompatible at $v$ or $vv_1$ and $v_1v_2$ are incompatible at $v_1$. Thus it follows by definition and the fact that $\frac{1}{n}\ll\mu\ll \eta, d,\frac{1}{r},\frac{1}{h}$, there exists a subfamily $\mathcal{K}'_v$ of
at least \[c(r-1,\frac{d}{2},h')\prod_{i=1}^{r-1} |U_i'|^{h_i}-2\mu n^{\sum_{i=1}^{r-1} h_i}\ge \frac{1}{2}c(r-1,\frac{d}{2},h')\prod_{i=1}^{r-1} |U_i'|^{h_i}\] copies of $K_{r-1}(h_1,\ldots,h_{r-1})$ that can form good pairs with $v$. Let $\mathcal{K}':=\bigcup_{v\in U_r'}\mathcal{K}'_v$. Since $|U_r'|\ge(1-(r-1)\varepsilon_r)|U_r|$, there are at least 
\begin{align}\label{eq1}
   &|U_r'|\cdot\min\{|\mathcal{K}'_v|: v\in U_r'\}|\ge(1-(r-1)\varepsilon_r)|U_r|\cdot \frac{1}{2}c(r-1,\frac{d}{2},h')\prod_{i=1}^{r-1} |U_i'|^{h_i} \nonumber \\
\ge&\frac{1}{4}c(r-1,\frac{d}{2},h')|U_r|\prod_{i=1}^{r-1} |U_i'|^{h_i}=c_1|U_r|\prod_{i=1}^{r-1} |U_i|^{h_i}
\end{align}
good pairs $(v,K)$ with center $v\in U_r'$ and $K\in \mathcal{K}'$, where $c_1=\frac{1}{4}c(r-1,\frac{d}{2},h')\cdot(d-\eps_r)^{\sum_{i=1}^{r-1} h_i}$.

Now we shall greedily embed compatible copies of $K_{r}(h_1,\ldots,h_{r})$ by extending every element inside $\mathcal{K}'$.
Here we define \[\mathcal{X}_1:=\{K\in \mathcal{K}'\mid K~\text{is contained in at least}~c_1 |U_r|/2~\text{good pairs with centers in}~U_r'\}.\]
\begin{claim}\label{cl1}
The following statements hold.
  \begin{enumerate}
     \item [(1)] $|\mathcal{X}_1|\ge \frac{c_1}{2}\prod_{i=1}^{r-1}|U_i|^{h_i}$;
     \item [(2)] every compatible copy $K\in \mathcal{X}_1$ is contained in at least $c_2 |U_r|^{h_r}$ compatible copies of $K_{r}(h_1,\ldots,h_{r})$ in $G[U_1,\ldots, U_{r}]$, where $c_2$ will be determined later.
   \end{enumerate}
\end{claim}
\begin{proof}
  We first show (2). For every $K\in \mathcal{X}_1$, we denote by $C(K)$ the set of vertices $v\in U_r$ which together with $K$ forms a good pair. Then $|C(K)|\ge c_1 |U_r|/2$. To extend a fixed $K\in \mathcal{X}_1$ into a compatible copy of $K_{r}(h_1,\ldots,h_{r})$, we can iteratively choose $h_r$ proper vertices $v_1,\ldots, v_{h_r}$ from $C(K)$ in the following way:
\begin{itemize}
  \item choose $v_1\in C(K)$ such that $V(K)\cup \{v_1\}$ induces a compatible copy of $K_r(h_1,\ldots,h_{r-1},1)$;
  \item for $2\le i\le h_r$, suppose that $v_1,\ldots,v_{i-1}$ have been given, then choose $v_i\in C(K)$ such that $V(K)\cup \{v_1,\ldots,v_i\}$ induces a compatible copy of $K_r(h_1,\ldots,h_{r-1},i)$.
\end{itemize}

Note that there are at least $c_1|U_r|/2-2\mu n\cdot e(K_{r-1}(h_1,\ldots,h_{r-1}))\ge c_1|U_r|/4$ choices for $v_1$, where the term $2\mu n\cdot e(K_{r-1}(h_1,\ldots,h_{r-1}))$ bounds the number of vertices $u\in C(K)$ for which there exists an incompatible pair $(uv,vw)$ of edges with $vw\in E(K)$. Similarly, there are at least \[c_1|U_r|/2-(i-1)-2\mu n\cdot e(K_{r-1}(h_1,\ldots,h_{r-1}))-\mu n\cdot (i-1)\sum_{i=1}^{r-1}h_i\ge c_1|U_r|/4\] choices for $v_i$.
Hence, there are at least $(c_1|U_r|/4)^{h_r}/h_r!=c_2|U_r|^{h_r}$ compatible copies of $K_{r}(h_1,\ldots,h_{r})$ in $G[U_1,\ldots, U_{r}]$ containing $K$, where $c_2=(c_1/4)^{h_r}/h_r!$.\medskip

The proof of (1) is easy by the double counting and we write $\mathcal{X}_2=\mathcal{K}'\setminus\mathcal{X}_1$. By inequality \eqref{eq1}, the number $N$ of good pairs $(v,K)$ with centers $v\in U_r'$ and $K\in \mathcal{K}'$ satisfy \[c_1|U_r|\prod_{i=1}^{r-1} |U_i|^{h_i}\le N\le |\mathcal{X}_2|\cdot c_1 |U_r|/2+|\mathcal{X}_1|\cdot|U_r|,\] which implies that $|\mathcal{X}_1|\ge \frac{c_1}{2}\prod_{i=1}^{r-1}|U_i|^{h_i}.$
\end{proof}

Thus, the proof easily follows from Claim~\ref{cl1} by letting $c(r,d,h)=c_1c_2/2$ and $\varepsilon^\ast=\varepsilon_r$ as there are at least \[|\mathcal{X}_1|\cdot c_2 |U_r|^{h_r}\ge \frac{c_1c_2}{2} \prod_{i=1}^{r}|U_i|^{h_i}=c(r,d,h)\prod_{i=1}^{r}|U_i|^{h_i}\] compatible copies of $K_{r}(h_1,\ldots,h_{r})$ in $G[U_1,\ldots,U_r]$, each containing exactly $h_i$ vertices in $U_i$ for every $i\in[r]$.


\end{proof}

\section{Building an absorbing set}\label{sec4}

The construction of an absorbing set has now been known via a novel idea of Montgomery~\cite{Montgomery} in a series of works \cite{chang2021}, provided that (almost) every set of $h$ vertices has linearly many vertex-disjoint absorbers. In this paper we make use of a `compatible' variant as follows, whose proof
follows from that of Lemma 2.2 in \cite{NP} and will be included in Appendix~\ref{A1}.

\begin{lemma}\label{abs set}
Let $t\in \mathbb{N}$, $H$ be an $h$-vertex graph and $\gamma>0$. Then there exists $\xi=\xi(t,h,\gamma)>0$ such that the following holds. If $(G,\mathcal{F})$ is an incompatibility system with $|V(G)|=n$ such that for every $S\in {V(G)\choose h}$ there is a family of at least $\gamma n$ vertex-disjoint $(H,t)$-absorbers, then $G$ contains a $\xi$-absorbing set of size at most $\gamma n$.
\end{lemma}

\subsection{Finding absorbers}\label{sec4.1}
As aforementioned, the first main task is to find for every $h$-set of vertices linearly many vertex-disjoint absorbers. To achieve this, we first state a crucial notion of $H$-reachability as follows, which is slightly different from the original version by Lo and Markstr\"om~\cite{2015Lo}.

\begin{defn}[Reachability \& Closedness]
Let $m,t\in \mathbb{N}$, $H$ be an $h$-vertex graph and $G$ be an $n$-vertex graph.
We say that two vertices $u,v\in V(G)$ are \emph{$(H,m,t)$-reachable} if for any set $W$ of $m$ vertices, there is a set $S\subseteq V(G)\setminus W$ of size at most $ht-1$ such that both $G[S\cup\{u\}]$ and $G[S\cup\{v\}]$ have compatible $H$-factors, where we call such $S$ an \emph{$H$-connector} for $u,v$. We say $V(G)$ is \emph{$(H,m,t)$-closed} if every two vertices $u,v$ in $V(G)$ are $(H,m,t)$-reachable. 
\end{defn}

The following result guarantees that every $h$-vertex set has a family of linearly many vertex-disjoint absorbers as long as $V(G)$ is $(H,\beta n,t)$-closed for some $\beta>0$ and $t\in \mathbb{N}$. Its proof follows from that of Lemma 3.9 in \cite{HMWY} and will be included in Appendix~\ref{A1} for completeness.

\begin{lemma}[Lemma 3.9 in \cite{HMWY}]\label{closed}
Given $h,t\in \mathbb{N}$ with $h\ge 3$ and $\beta>0$, the following holds for any $h$-vertex graph $H$ and sufficiently large $n \in \mathbb{N}$. If $(G,\mathcal{F})$ is an incompatibility system with $|V(G)|=n$ such that $V(G)$ is $(H,\beta n,t)$-closed, then every $S\in {V(G)\choose h}$ has at least $\frac{\beta}{h^3t}n$ vertex-disjoint $(H,t)$-absorbers.
\end{lemma}

Note that the original versions of Lemmas \ref{abs set} and \ref{closed} deal with absorbing sets and absorbers without the context of incompatibility systems. By going through their proofs, the two lemmas are also valid for compatible version. Based on Lemmas~\ref{abs set} and \ref{closed}, it suffices to show that $V(G)$ is closed.


\begin{lemma}\label{partition into B and U}
Let $H$ be an $h$-vertex graph and $\chi(H)\ge 2$. For any $\alpha>0$, there exist $\mu,\beta>0$ and $t\in \mathbb{N}$ such that for any sufficiently large $n$, if $(G,\mathcal{F})$ is an $\left(n,1-\frac{1}{\chi^*(H)}+\alpha,\mu\right)$-incompatibility system, then $V(G)$ is $(H,\beta n,t)$-closed.
\end{lemma}

\begin{proof}[\textbf{Proof of Lemma \ref{abs}}]
Given $\alpha,\sigma>0$ and an $h$-vertex graph $H$, we choose
\[\frac{1}{n}\ll \mu,\xi\ll \beta,\frac{1}{t}\ll \alpha,\sigma, \frac{1}{h},\] and let $(G,\mathcal{F})$ be an $\left(n,1-\frac{1}{\chi^*(H)}+\alpha,\mu\right)$-incompatibility system.
By Lemma \ref{partition into B and U}, $V(G)$ is $(H,\beta n,t)$-closed. By applying Lemma \ref{closed} to $G$, every $S\in {V(G)\choose h}$ has at least $\frac{\beta}{h^3t}n$ vertex-disjoint $(H,t)$-absorbers. Thus Lemma \ref{abs set} applied to $G$ with $\gamma=\min\{\frac{\beta}{h^3t}, \sigma\}$ immediately gives a $\xi$-absorbing set $A$ of size at most $\gamma n$. 

\end{proof}

\subsection{Lattice-based absorbing method}\label{sec4.2}
The proof of Lemma~\ref{partition into B and U} makes use of the lattice-based absorbing method developed by Han~\cite{Han2017Decision}, Keevash, Knox and Mycroft~\cite{kkmpoly15}. To illustrate this, we first need some notions introduced in \cite{kkmpoly15}. For any vector $\mathbf{x} = (x_{1},\ldots,x_{k} )\in \mathbb{Z}^{k}$, let $|\mathbf{x}| := \sum_{i=1}^{k}x_{i}$. We say that $\mathbf{x} \in \mathbb{Z}^{k}$ is an \emph{$r$-vector} if it has non-negative coordinates and $|\mathbf{x}| = r$.

\begin{defn}
Let $(G,\mathcal{F})$ be an incompatibility system with $|G|=n$ and $\mathcal{P} = \{V_{1},\ldots,V_{k}\}$ be a vertex partition of $V(G)$ for some integer $k \ge 1$.
  \begin{enumerate}
     \item The \emph{index vector} $\mathbf{i}_{\mathcal{P}}(S) \in \mathbb{Z}^{k}$ of a subset $S \subseteq V(G)$ with respect to $\mathcal{P}$ is the vector whose $i$th coordinate is the size of the intersection of $S$ with $V_i$, i.e. $|S\cap V_i|$, for every $i\in[k]$.
     \item Given an $h$-vertex graph $H$ and a constant $\beta > 0$, an $h$-vector $\mathbf{i} \in \mathbb{Z}^{k}$ is called $(H,\beta)$-\emph{robust} if for any set $W\subseteq V(G)$ of size $\be n$, $G-W$ contains a compatible copy of $H$ whose vertex set has the index vector $\mathbf{i}$ with respect to $\mathcal{P}$.
     \item Denote by $I_{\mathcal{P}}^{\beta}(G)$ the set of all $(H,\beta)$-robust $h$-vectors $\mathbf{i} \in \mathbb{Z}^{k}$. Let $L_{\mathcal{P}}^{\beta }(G)$ denote the lattice (additive subgroup) in $\mathbb{Z}^{k}$ generated by the elements of $I_{\mathcal{P}}^{\beta}(G)$.
     \item For every $j\in [k]$, let $\mathbf{u}_{j} \in \mathbb{Z}^{k}$ be the $j$th unit vector, i.e. $\mathbf{u}_j$ has $1$ on the $j$th coordinate and 0 on the other coordinates.
     \item A \emph{transferral} is a vector of the form $\mathbf{u}_{i}-\mathbf{u}_{j}$ for some distinct $i,j \in [k]$.
   \end{enumerate}
\end{defn}

As stated in Lemma \ref{partition into B and U}, we shall prove that $V(G)$ is $(H,\beta n,t)$-closed for some $\beta>0$ and $t\in \mathbb{N}$.
Our proof adopts a less direct approach as follows and builds on the merging techniques in \cite{HMWY}:
\begin{description}
  \item [$(1)$] \textbf{$1$-reachability:} We first show that every vertex in $V(G)$ is $1$-reachable to linearly many vertices (see Lemma~\ref{1-reach});
  \item [$(2)$] \textbf{Reachablility partition:} Then construct a partition $\mathcal{P}=\{V_1,\ldots,V_p\}$ of $V(G)$ such that every $V_i$ is closed (see Lemma~\ref{partition lem});
  \item [$(3)$] \textbf{Merging process:} Lemma~\ref{transferral lem} allows us to merge two parts $V_i,V_j$ into a larger (still closed) part, provided that there exists a transferral $\mathbf{u}_{i}-\mathbf{u}_{j}$ in the lattice $L_{\mathcal{P}}^{\beta}(G)$ for some $\beta>0$. Then the arguments often reduce
to iteratively detecting the existence of a transferral with respect to the up-to-date partition.
\end{description}

Here we collect a couple of lemmas as sketched above which together decode the proof of Lemma \ref{partition into B and U}.

\begin{lemma}\label{1-reach}
Let $H$ be an $h$-vertex graph and $\chi(H)=r\ge 2$. For any $\alpha>0$, there exist $\mu,\beta_1,\gamma_1>0$ such that for any sufficiently large $n$, if $(G,\mathcal{F})$ is an $\left(n,1-\frac{1}{r-1}+\alpha,\mu\right)$-incompatibility system, then every vertex $u\in V(G)$ is $(H,\beta_1 n,1)$-reachable to at least $\gamma_1 n$ other vertices.
\end{lemma}

\blemma[see Lemma 4.1 in \cite{HMWY}]\label{partition lem}
For positive constants $\gamma_1,\beta_1$, an integer $h \ge 2$ and an $h$-vertex graph $H$, there exist $\beta_2=\beta_2(\gamma_1,\beta_{1},h) > 0$ and $t\in\mathbb{N}$ such that the following holds for sufficiently large $n$. Let $(G,\mathcal{F})$ be an incompatibility system with $|G|=n$ such that every vertex in $V(G)$ is $(H,\beta_1 n,1)$-reachable to at least $\gamma_1 n$ other vertices. Then there is a partition $\mathcal{P} = \{V_{1},\ldots,V_{p}\}$ of $V(G)$ with $p \le \lceil\frac{1}{\gamma_1}\rceil$ such that for every $i \in [p]$, $V_{i}$ is $(H,\beta_2n,t)$-closed and has size at least $\frac{\gamma_1}{2} n$.
\elemma

\blemma\label{transferral lem} Given any positive integers $\ell,h,t$ with $h\ge 3$, constants $\beta,\mu >0$ and an $h$-vertex graph $H$, there exist $\be'>0$ and $t'\in \mathbb{N}$ such that the following holds for sufficiently large $n$. Let $(G,\mathcal{F})$ be an incompatibility system with $|G|=n$ and $\mathcal{P}=\{V_{1},\ldots,V_{\ell}\}$ be a partition of $V(G)$ such that every $V_{i}$ is $(H,\beta n,t)$-closed. If $\mathbf{u}_{i} - \mathbf{u}_{j} \in L_{\mathcal{P}}^{\mu }(G)$ for some distinct $i,j \in [\ell]$, then $V_{i} \cup V_{j}$ is $(H,\be' n ,t')$-closed.
\elemma

We shall present the proofs of Lemmas~\ref{1-reach} and \ref{transferral lem} in Sections~\ref{sec4.3} and \ref{sec4.4}, respectively, whilst the proof of Lemma~\ref{partition lem} follows that of Lemma~4.1 in \cite{HMWY} and will be included in Appendix~\ref{A2} for completeness.
\subsection{$V(G)$ is closed}

Now we are ready to prove Lemma \ref{partition into B and U}, taking for granted Lemmas~\ref{1-reach}-\ref{transferral lem}.
\begin{proof}[Proof of Lemma \ref{partition into B and U}]
Given $\alpha>0$ and an $h$-vertex graph $H$ with $\chi(H)=r\ge 2$, we choose \[\frac{1}{n}\ll\mu,\beta,\frac{1}{t}\ll \alpha,\frac{1}{h}\]
and let $(G,\mathcal{F})$ be an $\left(n,1-\frac{1}{\chi^*(H)}+\alpha,\mu\right)$-incompatibility system.
By Lemma \ref{1-reach}, we obtain that every vertex in $V(G)$ is $(H,\beta_1 n,1)$-reachable to at least $\gamma_1 n$ other vertices in $V(G)$. Then applying Lemma~\ref{partition lem} to $G$, we obtain a partition $\mathcal{P}_{0} = \{V_{1},\ldots,V_{p}\}$ of $V(G)$ for some integer $p \le \lceil\frac{1}{\gamma_1}\rceil$, where every $V_{i}$ is $(H, \beta_2n,t_2)$-closed and $|V_{i}|\ge \frac{\gamma_1n}{2}$.

We shall iteratively merge as many distinct parts as possible. As an intermediate step, we choose positive constants \[\be=\be_{p+1}\ll\be'_{p}\ll\be_{p}\ll\cdots\ll\be_3'\ll\be_3\ll\be_2'\ll\be_2 ~\text{and}~ \frac{1}{t}=\frac{1}{t_{p+1}}\ll\frac{1}{t_{p}}\cdots\ll\frac{1}{t_{3}}\ll\frac{1}{t_2}.\] Our merging procedure is stated as follows: at the step $s$ ($s\ge1$), based on the current partition $\mathcal{P}_{s-1} = \{V_1,\ldots,V_{p'}\}$ with $p'=p-s+1\in\mathbb{N}$ and every $V_i$ being $(H,\beta_{s+1}n,t_{s+1})$-closed, if there exist distinct $i,j\in [p']$ such that $\textbf{u}_i-\textbf{u}_j\in L^{\be'_{s+1}}_{\mathcal{P}_{s-1}}(G)$, then we merge $V_i,V_j$ as a new part. In fact, by applying Lemma~\ref{transferral lem} with $\beta=\beta_{s+1},\mu=\be'_{s+1}$ and $t=t_{s+1}$, we obtain that $V_i\cup V_j$ is $(H,\beta_{s+2}n,t_{s+2})$-closed. By renaming if necessary, we end the $s$th step with a new partition $\mathcal{P}_s= \{V_1,\ldots,V_{p'-1}\}$ such that every $V_i$ is $(H,\beta_{s+2}n,t_{s+2})$-closed. In this way, we continue until the procedure terminates after at most $p-1$ steps.

Suppose we end up with a final partition $\mathcal{P}=\mathcal{P}_s= \{V_1,\ldots,V_{p'}\}$ for some integer $p'=p-s$ and $0\le s<p$, where every $V_i$ is $(H,\beta_{s+2}n,t_{s+2})$-closed. Now we claim that $p'=1$, and thus the proof is completed as $V(G)$ is $(H,\beta n,t)$-closed by taking $\beta=\beta_{p+1}$ and $t=t_{p+1}$.
Assume for a contradiction that $p'>1$, and recall that every $V_i$ is $(H,\beta_{s+2}n,t_{s+2})$-closed with $|V_i|\ge \frac{\gamma_1}{2}n, i\in[p']$. 
We additionally choose \[\be'_{s+2}\ll\frac{1}{M}\ll\eps\ll d\ll\alpha,\gamma_1,\frac{1}{h}.\]
Then by applying Lemma~\ref{reg lem} to $G$ with constants $\eps$ and $d:=\frac{\alpha}{4}$, we obtain, as a refinement of $\mathcal{P}$, a partition $\mathcal{P}^*=\{V_0\}\cup\{V_{i,j}\subseteq V_i:i\in[p'],j\in[k_i]\}$ and a spanning subgraph $G'\subseteq G$ such that $|V_0|\le \eps n$ and $\frac{1}{\eps}\le \sum_{i=1}^{p'} k_i\le M$. We write $m:=|V_{i,j}|$ for all $i\in[p'],j\in[k_i]$ and $k:=\sum_{i=1}^{p'} k_i$. Let $R:=R(\eps,d)$ be the reduced graph for the partition $\mathcal{P}^*$. By the choice $\varepsilon\ll d \ll\alpha$ and the fact that $\delta(G)\ge(1-\frac{1}{\chi^*(H)}+\alpha)n$, we have $\delta(R)\ge \left(1-\frac{1}{\chi^*(H)}+\alpha-d-2\varepsilon\right)k\ge \left(1-\frac{1}{\chi^*(H)}+\frac{\alpha}{2}\right)k$.

We call a subgraph in $R$ \emph{crossing} with respect to $\mathcal{P}$ if it has two vertices in $R$ whose corresponding clusters lie in different parts of $\mathcal{P}$.
Recall that $p'>1$. Then the following claims allow us to further merge $V_i,V_j$ into a new part, a contrary to the minimality of $p'$. We divide the proof into two cases depending on the dichotomy of $\chi^*(H)$. Before that, we state a simple fact as follows which would be frequently used in the proof.
\begin{fact}\label{commom neighbor}
Let $G$ be an $n$-vertex graph and $v_1,\ldots,v_k$ be $k$ distinct vertices in $G$. Then $$|\cap_{i=1}^k N(v_i)|\ge \Sigma_{i=1}^k |N(v_i)|-(k-1)n.$$
\end{fact}

\begin{claim}
If $\chi^*(H)=\chi(H)=r$, then there exist distinct $i,j\in[p']$ such that $\textbf{u}_i-\textbf{u}_j\in L^{\be'_{s+2}}_{\mathcal{P}}(G)$.
\end{claim}
\begin{proof}

In this case, we have that $\delta(R)\ge \left(1-\frac{1}{r}+\frac{\alpha}{2}\right)k> \frac{k}{2}$.
Thus we can easily find in $R$ a crossing edge, say $V_{i,1}V_{j,1}$ for distinct $i,j\in [p']$. Next we show that such a pair $\{i,j\}$ is as desired by finding two vectors \[\mathbf{s},~\mathbf{t}\in ~ I^{\be'_{s+2}}_{\mathcal{P}}(G)~~\text{with}~~\mathbf{s}-\mathbf{t}~=~\mathbf{u}_i-\mathbf{u}_j.\]
By Fact~\ref{commom neighbor}, we can find a copy of $K_{r+1}$ containing $V_{i,1},V_{j,1}$ in $R$, where the other $r-1$ vertices, without loss of generality, are denoted as $V_{a_1,1},\ldots,V_{a_{r-1},1}$.

Let $h_1\le \ldots \le h_r$ be the color-class sizes of $H$ under a proper $r$-coloring of $V(H)$. Clearly, $H$ is a subgraph of $K_r(h_1,\ldots,h_r)$. By Lemma~\ref{K_h^r} applied to $G$ with
 \[\eta=\frac{1-\eps}{k},~U_s=V_{a_s,1}~\text{for}~ s\in[r-1] ~\text{and}~U_r=V_{j,1}\]
and the fact that $\frac{1}{n}\ll\mu\ll \eps\ll d$, we obtain that there exist at least $c(r,d,h)m^{h}$ compatible copies of $K_r(h_1,\ldots,h_r)$ (also $H$) in $G[V_{a_1,1},\ldots,V_{a_{r-1},1},V_{j,1}]$, each containing exactly $h_s$ vertices in $V_{a_s,1}$ for every $s\in[r-1]$ and $h_r$ vertices in $V_{j,1}$. Note that all those compatible copies of $H$ have the same index vector with respect to $\mathcal{P}$, denoted by $\mathbf{t}$. Similarly, Lemma~\ref{K_h^r} applied to $G$ with
\[\eta=\frac{1-\eps}{k},~U_s=V_{a_s,1}~\text{for}~s\in[r-1] ~\text{and}~U_r=V_{j,1},U_{r+1}=V_{i,1}\]
gives at least $c(r+1,d,h)m^{h}$ compatible copies of $K_{r+1}(h_1,\ldots,h_r-1,1)$ (also $H$ and here we may assume $h_r>1$ as the case $h_r=1$ can be done in the same way) each containing exactly $h_s$ vertices in $V_{a_s,1}$ for every $s\in[r-1]$, $h_r-1$ vertices in $V_{j,1}$ and one vertex in $V_{i,1}$, where the corresponding index vector is denoted as $\mathbf{s}$. Then it is easy to see that $\mathbf{s}-\mathbf{t}=\textbf{u}_i-\textbf{u}_j$ and we claim that $\mathbf{s},\mathbf{t}\in I^{\be'_{s+2}}_{\mathcal{P}}(G)$. Indeed, this easily follows by the choice $\be'_{s+2}\ll\frac{1}{M}\ll \eps\ll d\ll\alpha,\frac{1}{h}$ and the fact that there are at least $c'm^h$ compatible copies of $H$ for $c'=\min\{c(r,d,h),c(r+1,d,h)\}$, whose vertex sets all have index vector $\mathbf{s}$ (or $\mathbf{t}$). \medskip
\end{proof}

Before diving into the case $\chi^*(H)=\chi_{cr}(H)$ (i.e.~$\hcf(H)=1$), we first deal with a case which would be frequently used in the proof.
\begin{claim}\label{crossing}
If $\hcf_{\chi}(H)$ is finite and $R$ has a crossing copy of $K_r$, then there exist distinct $i,j\in[p']$ such that $\hcf_{\chi}(H)(\textbf{u}_i-\textbf{u}_j)\in L^{\be'_{s+2}}_{\mathcal{P}}(G)$.
\end{claim}
\begin{proof}
Without loss of generality, let $V_{x_1,1},\ldots,V_{x_{r},1}$ be the vertices of a crossing copy of $K_{r}$ in $R$, where $x_1=1,x_2=2$ and $1\le x_3\le \ldots\le x_{r}\le p'$ (if $r\ge 3$). Next we show that such a pair $\{1,2\}$ is as desired.

For any $r$-coloring $\phi$ in $\C(H)$ and its color-class sizes $h_1^{\phi}\le \cdots \le h_r^{\phi}$, by Lemma~\ref{K_h^r} applied to $G'$ with $U_i=V_{x_i,1}$ for every $i\in[r]$, we obtain for every permutation $\sigma\in \mathbf{S}_r$, a family $\mathcal{F}_{\sigma}$ of at least $cm^h$ compatible copies of $H$ each containing exactly $h_{\sigma(i)}^{\phi}$ vertices in $V_{x_i,1}$ for $i\in[r]$. It is easy to see by definition that every $\mathcal{F}_{\sigma}$ yields a $(H,\be'_{s+2})$-robust $h$-vector, denoted as $\mathbf{i}_{\sigma}^{\phi}$, whose $i$th coordinate is \[\mathbf{i}_{\sigma}^{\phi}(i)=\sum_{s:x_s=i}h_{\sigma(s)}^{\phi}, ~i\in [p'].\]
Thus for every $i\in [r-1]$, we choose $\sigma,\sigma'\in \mathbf{S}_r$ which have $\sigma(1)=i=\sigma'(2), \sigma(2)=i+1=\sigma'(1)$ and $\sigma(j)=\sigma'(j)$ for $j\neq1,2$. Then it follows that \[\mathbf{d}_i^{\phi}:=(h_{i+1}^{\phi}-h_i^{\phi}, h_i^{\phi}-h_{i+1}^{\phi}, 0,\ldots,0)~=~\mathbf{i}_{\sigma'}^{\phi}-\mathbf{i}_{\sigma}^{\phi}~\in~ L^{\be'_{s+2}}_{\mathcal{P}}(G).\]
Recall that $D(H)=\bigcup_{{\phi}\in\C(H)}\{h_{i+1}^{\phi}-h_i^{\phi}: i\in[r-1]\}$ and all the integers in $D(H)$ have the highest common factor $\hcf_{\chi}(H)$. Then there exist integers $a_i^{\phi}$ for any ${\phi}\in\C(H),i\in[r-1]$ such that \[\sum_{{\phi}\in\C(H),i\in[r-1]}a_i^{\phi}~(h_{i+1}^{\phi}-h_i^{\phi})~=~\hcf_{\chi}(H).\]
Therefore \[\sum_{{\phi}\in\C(H),i\in[r-1]}a_i^{\phi}~\mathbf{d}_i^{\phi}~=~(\hcf_{\chi}(H),-\hcf_{\chi}(H),0,\ldots,0)\] and by definition we have $\hcf_{\chi}(H)(\mathbf{u}_1-\mathbf{u}_2)\in L^{\be'_{s+2}}_{\mathcal{P}}(G)$. This completes the proof.
\end{proof}

Now we are ready to study the case $\chi^*(H)=\chi_{cr}(H)$. If $r\ge 3$, then by definition, it follows that $\hcf_{\chi}(H)=1$. Moreover, since $\chi_{cr}(H)>r-1$, it holds that $\delta(R)\ge \left(1-\frac{1}{r-1}+\frac{\alpha}{2}\right)k> \frac{k}{2}$ and we can easily find in $R$ a crossing edge, say $V_{1,1}V_{2,1}$ for example.
By Fact~\ref{commom neighbor} with $k=r-1$, we can similarly find a crossing copy of $K_{r}$ containing $V_{1,1},V_{2,1}$, and we are done by Claim~\ref{crossing}. Hence, it remains to consider the case when $H$ is bipartite with $\hcf_c(H)=1$ and $\hcf_{\chi}(H)\le 2$.

\begin{claim}
If $\chi(H)=2$, $\hcf_c(H)=1$ and $\hcf_{\chi}(H)\le 2$, then there exist distinct $i,j\in[p']$ such that $\textbf{u}_i-\textbf{u}_j\in L^{\be'_{s+2}}_{\mathcal{P}}(G)$.
\end{claim}
\begin{proof}
Since $\hcf_c(H)=1$, one can easily observe that $H$ is disconnected. Let $C_1,C_2,\ldots,C_{\ell}$ be the components of $H$ and $c_i:=|C_i|, i\in[\ell]$.
We further split the argument into two subcases: $\hcf_{\chi}(H)=1, 2$.

If $\hcf_{\chi}(H)=1$, then we may further assume that $R$ contains no crossing edge with respect to $\mathcal{P}$ as otherwise we are done by Claim~\ref{crossing}. Since $p'\ge 2$, we can pick two edges from different parts, say $V_{1,1}V_{1,2}$ and $V_{2,1}V_{2,2}$. Clearly, every component $C_i$ is a subgraph of $K_{h,h}$. By Lemma~\ref{K_h^r} to $V_{1,1},V_{1,2}$ (resp.~$V_{2,1},V_{2,2}$), we obtain for every subset $S\subseteq[\ell]$, a family $\mathcal{F}_{S}$ of at least $c^2m^{h}$ compatible copies of $H$ each containing exactly $\sum_{i\in S}c_i$ vertices in $V_1$ and $h-\sum_{i\in S}c_i$ vertices in $V_2$. It is easy to see by definition that every $\mathcal{F}_{S}$ yields an $(H,\beta_{s+2}')$-robust $h$-vector, denoted as $\mathbf{i}_{S}:=(\sum_{i\in S}c_i,h-\sum_{i\in S}c_i, 0,\ldots,0)$. Thus for every $i\in [\ell]$, we choose $S=\{i\},S'=\emptyset$ and it follows that \[\mathbf{d}_i:=(c_i, -c_i, 0,\ldots,0)~=~\mathbf{i}_S-\mathbf{i}_{S'}~\in~ L^{\be'_{s+2}}_{\mathcal{P}}(G).\]
Recall that all the integers in $\{c_1,c_2,\ldots, c_{\ell}\}$ have the highest common factor $\hcf_{c}(H)=1$. Then there exist integers $a_i$, $i\in[\ell]$ such that $\sum_{i\in[\ell]}a_ic_i=1$.
Therefore \[\sum_{i\in[\ell]}a_i\mathbf{d}_i=(1,-1,0,\ldots,0)\] and thus by definition we have $\mathbf{u}_1-\mathbf{u}_2\in L^{\be'_{s+2}}_{\mathcal{P}}(G)$.

For the case $\hcf_{\chi}(H)=2$, we may instead assume that there exists a crossing edge in $R$. As otherwise we can pick two edges from different parts since $p'\ge 2$, one can thus follow the same arguments in the paragraph above. Let $V_{1,1}V_{2,1}$ be such a crossing edge. Then Claim~\ref{crossing} immediately implies that $2(\mathbf{u}_1-\mathbf{u}_2)\in L^{\be'_{s+2}}_{\mathcal{P}}(G)$.

Since $\chi(H)=2$, $\hcf_{\chi}(H)=2$ and $\hcf_c(H)=1$, there exists an odd component of $H$, say $C_1$, i.e.~$c_1$ is odd. Given a $2$-coloring $\phi$ in $\C(H)$, we write $x_1, x_2$ as the color-class sizes for $C_1$, and $y_1, y_2$ for $H-C_1$ under $\phi$. Recall that $V_{1,1}V_{2,1}$ is a crossing edge. Then by Lemma~\ref{K_h^r} applied to $G$ with $r=2,U_1=V_{1,1}, U_2=V_{2,1}$ and $h_1=x_1+y_1,h_2=x_2+y_2$, we obtain at least $cm^h$ compatible copies of $H$ each containing exactly $x_i+y_i$ vertices in $V_{i,1}$ for $i\in[2]$. It is easy to see by definition that every such family yields an $(H,\be'_{s+2})$-robust $h$-vector, denoted as $\mathbf{i}_{even}:=(x_1+y_1,x_2+y_2, 0,\ldots, 0)$. Similarly we obtain $\mathbf{i}_{odd}:=(x_2+y_1,x_1+y_2, 0,\ldots, 0)\in I^{\be'_{s+2}}_{\mathcal{P}}(G)$.

Recall that $2(\mathbf{u}_1-\mathbf{u}_2)\in L^{\be'_{s+2}}_{\mathcal{P}}(G)$. Since $x_1, x_2$ have distinct parity, by letting $s=\frac{x_1-x_2-1}{2}$, we obtain that
\[\mathbf{i}_{even}-\mathbf{i}_{odd}-2s(\mathbf{u}_1-\mathbf{u}_2)= \mathbf{u}_1-\mathbf{u}_2~\in~ L^{\be'_{s+2}}_{\mathcal{P}}(G).\]
This completes the proof.
\end{proof}
\end{proof}




\subsection{1-reachablility}\label{sec4.3}


In this section, we shall focus on proving Lemma~\ref{1-reach}. 

\begin{proof}[Proof of Lemma \ref{1-reach}]

Given $\alpha>0$, we choose $$\frac{1}{n}\ll\mu,\beta_1\ll c,\gamma_1\ll \frac{1}{M}\ll\varepsilon\ll d\ll \alpha,\frac{1}{h}.$$
Apply Lemma \ref{reg lem} to $G$ and obtain a partition $V(G)=V_0\cup V_1\cup\ldots \cup V_k$ for some $1/\varepsilon\le k\le M$ such that $|V_0|\le \eps n,|V_i|=m$ for every $i\in[k]$. Let $R=R(\varepsilon,d)$ be the reduced graph for this partition. Then $\delta(R)\ge \left(1-\frac{1}{r-1}+\frac{\alpha}{2}\right)k$.

Choose an arbitrary vertex $u$ from $V(G)$.
Let $N'=\{V_i: i\in[k]~\text{and}~|N_G(u)\cap V_i|\ge \alpha m/2\}$.
By double counting, we have
$$(1-1/(r-1)+\alpha)n-|V_0|\le d_{G-V_0}(u)\le |N'|\cdot m+(k-|N'|)\cdot \alpha m/2,$$
that is, $|N'|\ge \left(1-\frac{1}{r-1}+\frac{\alpha}{4}\right)k$
.
Let $\ell=\left(1-\frac{1}{r-1}+\frac{\alpha}{4}\right)k$. We choose arbitrary $\ell$ clusters from $N'$.
Suppose that the corresponding vertices in $R$ of these $\ell$ clusters are $V_1,\ldots,V_{\ell}$. Write $L=\{V_1,\ldots,V_{\ell}\}$. Then we shall find a copy of $K_{r-1}$ in $R[L]$. Indeed, the case when $r=2$ is trivial and thus we
assume that $r\ge 3$. Since $\delta(R)\ge \left(1-\frac{1}{r-1}+\frac{\alpha}{2}\right)k$, we have $\delta(R[L])\ge \left(1-\frac{2}{r-1}+\frac{3\alpha}{4}\right)k$. Since $(r-2)\left(1-\frac{2}{r-1}+\frac{3\alpha}{4}\right)k-(r-3)\ell> 1$, by Fact \ref{commom neighbor} with $k=r-2$, we can greedily pick a copy of $K_{r-1}$ as every $r-2$ vertices in $L$ have a common neighbor in $L$. Suppose that the corresponding clusters of this copy of $K_{r-1}$ are $V_1,\ldots,V_{r-1}$. Again by Fact \ref{commom neighbor} applied to $R$ with $k=r-1$ and the fact that $(r-1)\left(1-\frac{1}{r-1}+\frac{\alpha}{2}\right)k-(r-2)k> 1$, we obtain that $V_1,\ldots,V_{r-1}$ have a common neighbor in $V(R)$, say $V_r$. It is worth noting that here the cluster $V_r$ may not in $L$. Let $V_i'=V_i\cap N_G(u)$ for every $i\in [r-1]$. Then $|V_i'|\ge \frac{\alpha m}{2}\ge \frac{\alpha(1-\varepsilon)}{2k}n$ for every $i\in[r-1]$.

Let $V'_r:=\{v\in V_r\mid d_{V_i'}(v)\ge (d-\eps)|V_i'| \text{ ~for~each~} i\in[r-1]\}$. Then by Fact~\ref{large deg}, we have that $|V_r'|\ge (1-(r-1)\eps)m\ge \frac{1}{2}m$. The following claim would finish the proof as $|V_r'|\ge \gamma_1 n$.
\begin{claim}\label{u-v reachable}
Every vertex $v\in V_r'$ is $(H,\beta_1n,1)$-reachable to $u$.
\end{claim}
To prove this, for every $v\in V_r'$, we write $U_i=V_i'\cap N(v)$ for every $i\in[r-1]$ and $U_r:=V_r'\setminus \{v\}$. Then $|U_i|\ge(d-\eps)|V_i'|\ge \frac{d\alpha}{4} m\ge \frac{d\alpha}{8M}n, i\in[r]$ and Lemma \ref{slicing lem} implies that $U_1,\ldots,U_r$ are pairwise $(\varepsilon',d')$-regular with $\varepsilon'=\max\{4\varepsilon/d\alpha,2\varepsilon\}$ and $|d'-d|<\varepsilon$. By Lemma \ref{K_h^r} to $G$ with all subsets $U_i$ as given above, $\eta=\frac{d\alpha}{8M}$ and the fact that $\mu\ll\frac{1}{M}\ll \eps\ll d\ll\alpha,\frac{1}{h}$, we obtain at least $cn^{rh-1}$ compatible copies of $K_r(h,\ldots,h,h-1)$ in $G[U_1,\ldots,U_r]$ each containing $h$ vertices in $U_i$ for $i\in[r-1]$ and $h-1$ vertices in $U_r$. Note that by Fact~\ref{fact1}, there are at most $4\mu n^{rh-1}\le \frac{c}{2}n^{rh-1}$ compatible copies of $K_r(h,\ldots,h,h-1)$, each of which fails to form a compatible copy of $K^h_r$ with $u$ or $v$. Therefore we can find a family $\mathcal{K}$ of at least $\frac{c}{2}n^{rh-1}$ compatible copies of $K_r(h,\ldots,h,h-1)$ which together with $u$ (also for $v$) form compatible copies of $K^h_r$. Moreover, from the choice $\beta_1\ll c$, one can observe that deleting an arbitrary set of $\beta_1 n$ vertices would destroy at most $\beta_1n^{rh-1}<|\mathcal{K}|$ copies of $K_r(h,\ldots,h,h-1)$ in $\mathcal{K}$. Thus, $u$ and $v$ are $(K^h_r,\beta_1 n,1)$-reachable and also are $(H,\beta_1 n,1)$-reachable.
\end{proof}

\subsection{Transferral}\label{sec4.4}
Following similar arguments as in the proof of Lemma 3.9 in~\cite{Han2017packing}, we give a proof of Lemma~\ref{transferral lem}.
\begin{proof}[Proof of Lemma~\ref{transferral lem}]
Let $G$, $\be,\mu>0$ and $\mathcal{P}=\{V_1, V_2,\ldots,V_{\ell}\}$ be a partition of $V(G)$ as given in the assumption.
Recall that $I=I_{\mathcal{P}}^{\mu}(G)\subseteq \mathbb{Z}^{\ell}$ is a base for $L_{\mathcal{P}}^{\mu}(G)$, that is, every vector in $L_{\mathcal{P}}^{\mu}(G)$ can be written as linear combinations of the $h$-vectors in $I$. Thus given $\mathbf{u}_i -  \mathbf{u}_j\in L_{\mathcal{P}}^{\mu}(G)$, there exist $a_{\mathbf{v}}\in \mathbb{Z}$ such that $\mathbf{u}_i -  \mathbf{u}_j = \sum_{\mathbf{v}\in I}a_{\mathbf{v}}\mathbf{v}$. Without loss of generality, we may take $i=1,j=2$ for instance. For every $\mathbf{v}\in I$,  let $p_{\mathbf{v}}=\max\{a_{\mathbf{v}}, 0\}$ and $q_{\mathbf{v}}=\max\{-a_{\mathbf{v}}, 0\}$. Hence
\begin{align}
 \mathbf{u}_1 -  \mathbf{u}_2=\sum_{\mathbf{v}\in I} (p_{\mathbf{v}} - q_{\mathbf{v}}) \mathbf{v},
\quad i.e. \quad
\sum_{\mathbf{v}\in I}{q_{\mathbf{v}}}\mathbf{v} + \mathbf{u}_1=\sum_{\mathbf{v}\in I} {p_{\mathbf{v}}}\mathbf{v} +  \mathbf{u}_2. \label{regroup11}
\end{align}
By comparing the sums of all the coordinates from two sides of either equation in \eqref{regroup11}, we obtain that
$
\sum_{\mathbf{v}\in I}{p_{\mathbf{v}}}=\sum_{\mathbf{v}\in I}{q_{\mathbf{v}}}=:C$.

Next, by choosing $\frac{1}{n}\ll\be',\frac{1}{t'}\ll \be,\mu,\frac{1}{C}, \frac{1}{t},\frac{1}{h}$, we shall show that every two vertices $x\in V_1$ and $y\in V_2$ are $(H,\be'n,t')$-reachable. Let $x$ and $y$ be given as above and $W\subset V(G)\setminus \{x,y\}$ be any vertex set of size $\be'n$. Since every $h$-vector $\mathbf{v}\in I$ is $(H,\mu n)$-robust and $\be'\ll \mu$, we can greedily select, outside $W\cup\{x,y\}$, a collection of $p_{\mathbf{v}} + q_{\mathbf{v}}$ vertex-disjoint compatible copies of $H$ with index vector $\mathbf{v}$ for every $\mathbf{v}\in I$. This is possible as during the process we need to avoid at most $|W|+2+h\sum_{\mathbf{v}\in I}(p_{\mathbf{v}} + q_{\mathbf{v}})< \frac{\mu}{2} n$ vertices. This gives rise to two disjoint families $\mathcal{F}^p$ and $\F^q$, where $\F^p$ consists of $p_{\mathbf{v}}$ vertex-disjoint compatible copies of $H$ with index vector $ \mathbf{v}$ for every $ \mathbf{v}\in I$, and $\F^q$ consists of $q_{\mathbf{v}}$ vertex-disjoint compatible copies of $H$ with index vector $\mathbf{v}$ for every $ \mathbf{v}\in I$.

By equality \eqref{regroup11}, we have $|V(\F^p)|=|V(\F^q)|= hC$ and $\mathbf{i}_{}(V(\F^q))+ \mathbf{u}_1=\mathbf{i}_{}(V(\F^p)) +  \mathbf{u}_2$.
This implies that we may write $V(\F^p)=\{x_1, \dots, x_{hC}\}$, $V(\F^q)= \{y_1, \dots, y_{hC}\}$ such that
$x_1\in V_1$, $y_1\in V_2$, and for $i\ge 2$, $x_i$ and $y_i$ are from the same part of $\mathcal{P}$. Since each $V_i$ is $(H,\be n, t)$-closed for each $i\in [\ell]$, we greedily pick a collection $\{S_2,S_3,\ldots, S_{hC}\}$ of vertex-disjoint subsets in $V(G)\setminus (W\cup V(\F^p\cup\F^q)\cup\{x,y\})$ such that every $S_j$ is an $H$-connector for $x_j,y_j$ with $|S_j|\le ht-1$ for $2\le j\le hC$. Indeed, by the choice that $\frac{1}{n}\ll\be'\ll \be,\frac{1}{C},\frac{1}{h}$, we need to avoid at most
\[\left|\left(\bigcup_{j=2}^{hC}S_j\right)\cup W\cup V(\F^p\cup\F^q)\cup\{x,y\}\right|\le \frac{\be}{2} n\] vertices and therefore we are able to select every $S_j$ as $x_j$ and $y_j$ are $(H,\be n ,t)$-reachable for $2\le j\le hC$. Similarly, we additionally choose two $H$-connectors vertex-disjoint from each other and all other previously chosen vertices, say $S_0$ and $S_1$ for $x,x_1$ and $y,y_1$, respectively. At this point, we shall verify that the set $\hat S:=\bigcup_{j=0}^{hC}S_j\cup V(\F^p\cup\F^q)$ is an $H$-connector for $x,y$ of size at most $h(t+C+thC)-1$. In fact, to build a compatible $H$-factor for $G[\hat S \cup \{x\}] $ (leaving $y$ uncovered), we can take $H$-tilings in $G[S_0\cup\{x\}]$, $G[V(\F^p)]$ and $G[S_j\cup\{y_j\}]$ for $j\in [hC]$, which altogether form an $H$-tiling as desired. Similarly, one can observe that $G[\hat S \cup \{y\}]$ also admits a compatible $H$-factor. Hence $x$ and $y$ are $\left(H,\be'n, t'\right)$-reachable by taking $t'=t+C+thC$.

\end{proof}


\section{Acknowledgement}
The authors would like to thank Jie Han, Laihao Ding and Guanghui Wang for stimulating discussions and valuable comments.\\

\noindent\textbf{Note added before submission.} While preparing this paper, we learnt that Yangyang Cheng and Laihao Ding proved various Tur\'{a}n-type results in the incompatibility systems.

\bibliographystyle{abbrv}
\bibliography{ref}

\begin{appendices}

\section{Proofs of Lemma~\ref{abs set} and Lemma~\ref{closed}}\label{A1}

The proof of Lemma~\ref{abs set} can be easily derived by the following result.

\begin{lemma}\emph{\cite{chang2020factors}}\label{chang}
Let $F$ be a $k$-graph with $b$ vertices. Let $\gamma>0$ and $0<\varepsilon<\min\{1/3, \gamma/2\}$.
  Then there exists $\xi=\xi(b,\gamma)$ such that the following holds for every sufficiently large~$n$.
  Suppose $\mathcal{G}$ is an $n$-vertex $k$-graph such that, there exists $V_0\subset V(\mathcal{G})$ of size at most $\varepsilon n$ such that for every $b$-subset $S$ of $V(\mathcal{G})\setminus V_0$, there are at least $\gamma n$ vertex-disjoint $S$-absorbers, and for every vertex $v\in V(\mathcal{G})$, there are at least $\gamma n$ copies of $F$ containing $v$, where each pair of these copies only intersects at vertex $v$.
  Then $\mathcal{G}$ contains a subset $A\subseteq V(\mathcal{G})$ of size at most $\gamma n$ such that, for every subset
  $R\subseteq V(\mathcal{G})\setminus A$ with $|R|\leq \xi n$ such that $b$ divides $|A|+|R|$, the $k$-graph $\mathcal{G}[A\cup R]$ contains an
  $F$-factor.
\end{lemma}
\begin{proof}[Proof of Lemma~\ref{abs set}]

Given $h,t\in \mathbb{N}$, an $h$-vertex graph $H$ and $\gamma>0$, we shall choose $\frac{1}{n}\ll\xi\ll \frac{1}{t},\frac{1}{h},\gamma$. Let $(G,\mathcal{F})$ be an incompatibility system with $|V(G)|=n$ such that for every $S\in {V(G)\choose h}$ there is a family of at least $\gamma n$ vertex-disjoint $(H,t)$-absorbers. We construct an $h$-graph $\mathcal{G}$ on vertex set $V(G)$ where $E(\mathcal{G})$ consists of all compatible copies of $H$ in $(G,\mathcal{F})$. Let $F$ be an $h$-graph with exactly one edge. Then by definition, every $S\in {V(\mathcal{G})\choose h}$ has at least $\gamma n$ vertex-disjoint $(F,t)$-absorbers. By applying Lemma~\ref{chang} to $\mathcal{G}$ with $b=h$ and $V_0=\emptyset$, we obtain a subset $A\subseteq V(\mathcal{G})$ of size at most $\gamma n$ such that, for every subset
  $R\subseteq V(\mathcal{G})\setminus A$ with $|R|\leq \xi n$ such that $h$ divides $|A|+|R|$, the $k$-graph $\mathcal{G}[A\cup R]$ contains an $F$-factor.
Then $A$ is a desired $\xi$-absorbing set in $(G,\mathcal{F})$.
\end{proof}
\begin{proof}[Proof of Lemma~\ref{closed}]
For positive integers $h, t$ with $h\ge 3$ and  $\beta>0$, let $G,H,\mathcal{F}$ be given as in the assumption such that $V(G)$ is $(H,\beta n,t)$-closed. Then for every $h$-set $S\<V(G)$, we shall greedily construct as many pairwise disjoint $(H,t)$-absorbers for $S$ as possible. Let $\mathcal{A}=\{A_1,A_2,\ldots, A_{\ell}\}$ be a maximal family of vertex-disjoint $(H,t)$-absorbers as above. Suppose to the contrary that $\ell<\tfrac{\be}{h^3t}n,$.
Then $|\cup_{\mathcal{A}}A_i| \le \tfrac{\beta}{h} n $ as each such $A_i$ has size at most $h^2t$ (see Definition~\ref{defabs}). \medskip

As $V(G)$ is $(H,\beta n,t)$-closed and $|\cup_{\mathcal{A}}A_i\cup S|\le \beta n$, we can pick a compatible copy of $H$ in $V(G)\setminus(\cup_{\mathcal{A}}A_i\cup S)$ whose vertex set is denoted by $T$. Write $S=\{s_1,s_2,\ldots,s_h\}$ and $T=\{t_1, t_2,\dots, t_h\}$. We now greedily pick a collection $\{S_1,S_2,\ldots, S_{h}\}$ of vertex-disjoint subsets in $V(G)\setminus (\bigcup_{\mathcal{A}}A_i\cup S\cup T)$ such that each $S_i$ is an $H$-connector for $s_i,t_i$ with $|S_i|\le ht-1$. Indeed, since
\[\left|\bigcup_{i=1}^\ell A_i\cup\left(\bigcup_{i=1}^{k'}S_i\right)\cup S\cup T\right|\le\be n,\]
for any $0\leq k'\leq h$ (using that $n$ is sufficiently large), it is possible pick each $S_i$ one by one because $s_i$ and $t_i$ are $(H,\be n,t)$-reachable.
At this point, it is easy to verify that $\bigcup_{i=1}^{h}S_i\cup T$ is actually an $(H,t)$-absorber for $S$, contrary to the maximality of $\ell$.
\end{proof}

\section{Proof of Lemma~\ref{partition lem}}\label{A2}

We begin with a simple result as follows.
\begin{lemma}\label{concatenate}
Let $s_1,s_2\in \mathbb{N}, \beta_1,\beta_2>0$ and $H$ be an $h$-vertex graph. The following holds for any sufficiently large $n$. Let $(G,\mathcal{F})$ be an incompatibility system with $|V(G)|=n$, $u,v\in V(G)$ and $Z\subseteq V(G)$ with $|Z|=cn$ for some positive constant $c$. If $u$ is $(H,\beta_1n,s_1)$-reachable to every vertex in $Z$, while $v$ is $(H,\beta_2n,s_2)$-reachable to every vertex in $Z$, then $u,v$ are $(H,\beta_3 n,s_1+s_2)$-reachable, where $\beta_3 n=\min\{cn-1,\beta_1n/2,\beta_2n/2\}$.
\end{lemma}
\begin{proof}
Let $W$ be an arbitrary subset of $G$ with $|W|\le\beta_3 n$. Since $|W|\le \min\{cn-1,\beta_1n/2\}$, we can find a vertex $u_1\in Z\setminus W$ and an $H$-connector $S_1$ for $u,u_1$ in $V(G)\setminus(W\cup\{u,u_1,v\})$ with $|S_1|\le ht_1-1$. Meanwhile, since $|W|\le \beta_2n/2$, we can also find an $H$-connector $S_2$ for $u_1,v$ in $V(G)\setminus(W\cup\{u,u_1,v\}\cup S_1)$ with $|S_2|\le ht_2-1$. Hence, $u,v$ are $(H,\beta_3n,t_1+t_2)$-reachable as both $G[S_1\cup S_2\cup\{u_1,u\}]$ and $G[S_1\cup S_2\cup\{u_1,v\}]$ contain compatible $H$-factors.
\end{proof}



\begin{proof}[Proof of Lemma \ref{partition lem}]
Write $r_0 = \lceil1/\gamma_1\rceil+1$, $t_j = 2^{r_0+1-j}$ for each $j\in [r_0+1]$ and choose constants $\frac{1}{n}\ll\lambda_1\ll \lambda_2\ll \lambda_3\cdots\ll\lambda_{r_0+1}=\be_1,\gamma_1, \frac{1}{r_0}$.
Let $(G,\mathcal{F})$ be an incompatibility system with $|G|=n$ such that every vertex in $V(G)$ is $(H,\beta_1 n,1)$-reachable to at least $\gamma_1 n$ other vertices.

Assume that there are two vertices that are not $(H, \lambda_{2} n, t_2)$-reachable, as otherwise we just output
$\mathcal P=\{V(G)\}$ as the desired partition. 
Observe also that every set of $r_0$ vertices must contain two vertices that are $(H, \lambda_{r_0}n, t_{r_0})$-reachable to each other.
Indeed, by the inclusion-exclusion principle we obtain a pair of vertices, say $u,v$, which are $(H, \beta_1 n, 1)$-reachable to a set of at least $\delta' n$ vertices for some $\delta'\ge \gamma_1/\binom{r_0}{2}$.
Thus, by Lemma~\ref{concatenate} with $s_1=s_2=1$, we deduce that $u$ and $v$ are $(H, \lambda_{r_0}n, t_{r_0})$-reachable as  $\lambda_{r_0}\ll \beta_1,\gamma_1, \frac{1}{r_0}$ and $t_{r_0}=2$.

Let $d$ be the largest integer such that there are $d$ vertices $v_1,\dots, v_d$ in $G$ which are pairwise \emph{not} $(H, \lambda_d n, t_d)$-reachable. Note that $d$ exists and $2\le d\le r_0-1$. Indeed, we assumed above that there are  $2$ vertices that are not $(H,\lambda_2n,t_2)$-reachable and  if $d= r_0$, then there would be $r_0$ vertices in $G$ which are pairwise not $(H, \lambda_{r_0}n, t_{r_0})$-reachable, contrary to the observation above. Let $S=\{v_1, \dots, v_d\}$. Note that $v_1, \dots, v_d$ are also pairwise not $(H, \lambda_{d+1} n, t_{d+1})$-reachable.

For a vertex $v$ and every $j\in [r_0+1]$, we write $\tilde{N}_j(v)$ for the set
of vertices which are $(H, \lambda_j n, t_j)$-reachable to $v$. For the sets $\tilde{N}_{d+1}(v_i)$, $i\in [d]$, we conclude that
\begin{description}
  \item[(1)] any vertex $v\in V(G) \setminus \{ v_1, \dots, v_d\}$ must be
in $\tilde{N}_{d+1}(v_i)$ for some $i\in [d]$.
Otherwise, $v, v_1, \dots, v_d$ are pairwise not $(H, \lambda_{d+1} n, t_{d+1})$-reachable, contradicting the maximality of $d$.
  \item[(2)] $| \tilde{N}_{d+1}(v_i) \cap \tilde{N}_{d+1}(v_j)|< \lambda_{d+1} n$ for any distinct $i,j\in[d]$. Otherwise Lemma~\ref{concatenate} applied with $s_1=s_2=t_{d+1}$ implies that $v_i, v_j$ are $(H, \lambda_d n, t_d)$-reachable to each other, a contradiction to the choice of vertices $v_i,i\in[d]$ (using that  $\lambda_d\ll\lambda_{d+1}$ and $2t_{d+1}=t_d$).
\end{description}

For $i\in [d]$, let \[U_i = (\tilde{N}_{d+1}(v_i)\cup \{ v_i\})
\setminus \bigcup_{ j\not = i} \tilde{N}_{d+1}(v_j).\]
Then we claim that each $U_i$ is $(H, \lambda_{d+1}n, t_{d+1})$-closed.
Indeed otherwise, there exist $u_1, u_2\in U_i$ that are not
$(H, \lambda_{d+1}n, t_{d+1})$-reachable to each other. Then
$\{ u_1, u_2\} \cup \{ v_1, \dots, v_d\} \setminus \{v_i\}$
contradicts the maximality of $d$.

Let $U_0 = V(G) \setminus (U_1\cup\dots\cup U_d)$.
We have $|U_0|< d^2\lambda_{d+1} n$ due to {\textbf{(2)}} above. In order to obtain the desired reachability partition,
we now drop every vertex of $U_0$ back into some $U_i$ for $i\in [d]$ as follows. Since each $v\in U_0$ is $(H, \beta_{1}n, 1)$-reachable (i.e.~$(H, \lambda_{r_0+1}n, t_{r_0+1})$-reachable) to at least $\gamma_1 n$ vertices, it follows from the fact $\lambda_{d+1}\ll \gamma_1,\frac{1}{r_0}$ that
\[|\tilde{N}_{r_0+1}(v)\setminus U_0| \ge \gamma_1 n - |U_0|\ge \gamma_1 n - d^2\lambda_{d+1} n > d\lambda_{d+1} n.\]
Therefore by the pigeonhole principle there exists some $i \in [d]$ such that $v$ is $(H, \lambda_{r_0+1}n, t_{r_0+1})$-reachable to at least
$\lambda_{d+1} n+1$ vertices in $U_i$. 
Lemma~\ref{concatenate} (using that $\lambda_d\ll\lambda_{d+1}\ll\lambda_{r_0+1}$) implies  that  $v$ is $(H, \lambda_d n, t_d)$-reachable to every vertex in $U_i$.

Recall that $2\le d\le r_0-1$.
Now partition $U_0$ as  $U_0=\cup_{i\in d}R_i$ where for each $i\in[d]$,  $R_i$ denotes a set of vertices $v\in U_0$ that are $(H, \lambda_d n, t_d)$-reachable to every vertex in $U_i$. Again by Lemma~\ref{concatenate} with $s_1=s_2=t_d$, for every $i\in [d]$,  every two vertices in $R_i$ (if any) are $(H, \lambda_{d-1} n, t_{d-1})$-reachable to each other. Let $\mathcal{P}=\{V_1, \dots, V_d\}$ be the final partition by setting $V_i=R_i\cup U_i$. Then each $V_i$ is $(H, \lambda_{d-1} n, t_{d-1})$-closed.
Also, for each $i\in[d]$, it holds that \[|V_i| \ge |U_i| \ge |\tilde{N}_{d+1}(v_i)| - d^2 \lambda_{d+1} n
\ge |\tilde{N}_{r_0+1}(v_i)| - \tfrac{\gamma_1}{2} n
\ge  \tfrac{\gamma_1}{2}n.\]
This completes the proof by taking $\be_2=\lambda_{d-1}$ and $t=t_{d-1}$.
\end{proof}

\end{appendices}


\end{document}